\documentclass[11pt]{article}
\usepackage[top=1in, bottom=1in, left=1in, right=1in]{geometry}

\usepackage[linktocpage,colorlinks,linkcolor=blue,anchorcolor=blue,citecolor=blue,urlcolor=blue,pagebackref]{hyperref}
\usepackage{euscript,mdframed}
\usepackage{microtype,todonotes,relsize}
\usepackage{booktabs}
\usepackage{amsmath,amsthm}
\usepackage{algorithmic}

\usepackage{accents}
\usepackage{comment,url,graphicx,relsize}
\usepackage{amssymb,amsfonts,amsmath,amsthm,amscd,dsfont,mathrsfs,mathtools,nicefrac, bm}
\usepackage{float,psfrag,epsfig,color,xcolor,url}
\usepackage{epstopdf,bbm,mathtools,enumitem}
\usepackage{subfigure}
\usepackage[ruled,vlined]{algorithm2e}
\usepackage{tablefootnote}
\graphicspath{{Plots/}}
\usepackage{diagbox}
\usepackage{tikz}
\usetikzlibrary{calc,patterns,angles,quotes}
\usepackage{makecell}
\usepackage{multirow}

\def\norm#1{\left\|{#1}\right\|} % A norm with 1 argument
\newcommand{\E}{\mathbb{E}}
\newcommand{\RR}{\mathbb{R}}

\newcommand{\ones}{\mathbf{1}}

\providecommand{\diag}{\mathop\mathrm{diag}}

\newcommand{\X}{\mathbf{X}}
\newcommand{\Z}{\mathbf{Z}}
\newcommand{\V}{\mathbf{V}}
\newcommand{\U}{\mathbf{U}}
\newcommand{\A}{\mathbf{A}}

\ifdefined\nonewproofenvironments\else
\ifdefined\ispres\else
\renewenvironment{proof}{\noindent\textbf{Proof.}\hspace*{.3em}}{\qed\\}
\newenvironment{proof-sketch}{\noindent\textbf{Proof Sketch}
  \hspace*{0.em}}{\qed\bigskip\\}
\newenvironment{proof-idea}{\noindent\textbf{Proof Idea}
  \hspace*{0.em}}{\qed\bigskip\\}
\newenvironment{proof-of-lemma}[1][{}]{\noindent\textbf{Proof of Lemma {#1}.}
  \hspace*{0.em}}{\qed\\}
\newenvironment{proof-of-corollary}[1][{}]{\noindent\textbf{Proof of Corollary {#1}.}
  \hspace*{0.em}}{\qed\\}
\newenvironment{proof-of-theorem}[1][{}]{\noindent\textbf{Proof of Theorem {#1}.}
  \hspace*{0.em}}{\qed\\}
\newenvironment{proof-attempt}{\noindent\textbf{Proof Attempt}
  \hspace*{0.em}}{\qed\bigskip\\}

\fi
\newtheorem{theorem}{Theorem}[section]
\newtheorem{lemma}{Lemma}[section]

\newtheorem{assumption}{Assumption}[section]
\newtheorem{remark}{Remark}[section]

\newtheorem{definition}{Definition}[section]

\usetikzlibrary{calc}

\allowdisplaybreaks
\usepackage[authoryear,round]{natbib}
\renewcommand*{\backref}[1]{\ifx#1\relax \else Page #1 \fi}
\renewcommand*{\backrefalt}[4]{%
  \ifcase #1 \footnotesize{(Not cited.)}%
  \or        \footnotesize{(Cited on page~#2.)}%
  \else      \footnotesize{(Cited on pages~#2.)}%
  \fi
}

\newcommand*{\colorboxed}{}
\def\colorboxed#1#{%
  \colorboxedAux{#1}%
}
\newcommand*{\colorboxedAux}[3]{%
  \begingroup
    \colorlet{cb@saved}{.}%
    \color#1{#2}%
    \boxed{%
      \color{cb@saved}%
      #3%
    }%
  \endgroup
}
\numberwithin{equation}{section}
\usepackage{xcolor} 
\usepackage{marginnote}

\newcommand{\todol}[2][]{{%
 \let\marginpar\marginnote
 \reversemarginpar
 \renewcommand{\baselinestretch}{0.8}%
 \todo[color=yellow]{#2}}}

\title{Problem-Parameter-Free Decentralized Nonconvex Stochastic Optimization}
\author{
{Jiaxiang Li} \thanks{Department of Electrical and Computer Engineering, University of Minnesota, Twin Cities.  \texttt{li003755@umn.edu}}
\and
{Xuxing Chen} \thanks{Department of Mathematics, University of California, Davis.  \texttt{xuxchen@ucdavis.edu}}
\and
{Shiqian Ma} \thanks{Department of Computational Applied Math and Operations Research, Rice University. Research supported in part by NSF grants DMS-2243650, CCF-2308597, CCF-2311275 and ECCS-2326591, and a startup fund from Rice University. \texttt{sqma@rice.edu}}
\and
{Mingyi Hong} \thanks{Department of Electrical and Computer Engineering, University of Minnesota, Twin Cities.  \texttt{mhong@umn.edu}}
}
\date{}

\begin{document}
\maketitle

% REQUIRED
\begin{abstract}
Existing decentralized algorithms usually require knowledge of problem parameters for updating local iterates. For example, the hyperparameters (such as learning rate) usually require the knowledge of Lipschitz constant of the global gradient or topological information of the communication networks, which are usually not accessible in practice. In this paper, we propose D-NASA, the first algorithm for decentralized nonconvex stochastic optimization that requires no prior knowledge of any problem parameters. We show that D-NASA has the optimal rate of convergence for nonconvex objectives under very mild conditions and enjoys the linear-speedup effect, i.e. the computation becomes faster as the number of nodes in the system increases. Extensive numerical experiments are conducted to support our findings.
\end{abstract}

\section{Introduction} 
Decentralized (distributed) optimization appears in many applications, such as machine learning~\citep{lian2017can, tang2018d, lian2018asynchronous}, robotics~\citep{queralta2020collaborative}, signal processing~\citep{hong2015unified}, and control systems~\citep{nedic2018distributed, yang2019survey}. In machine learning, decentralized optimization arises naturally when the data is either stored in different physical locations, or split into different servers to boost training efficiency. Therefore the main concerns for decentralized algorithms are data privacy, algorithmic scalability and robustness. For example, starting from earlier works~\cite{lian2017can,tang2018d}, researchers seek to develop scalable decentralized algorithms for distributed training that are provably more efficient than centralized algorithms, usually reflected in an inverse dependency over the number of devices/nodes in their final convergence rate, known as \textit{linear speedup}. 

One obstacle of applying most of the developed decentralized algorithms in practice is that their hyperparameters (such as learning rate) usually depend on information of the problem in order to show a theoretical convergence, e.g, the Lipschitz constant of the global gradient, the spectral gap of the graph adjacency matrix or other topological information of the problem. Such information is usually hard to obtain due to either physical/privacy restrictions or computational constraints (e.g. due to excessive amount of data in machine learning applications), and tedious hyperparameter tuning is thus required. Nonetheless, in most of these works people demonstrate a decent performance in experiments of the proposed algorithms without strictly following the hyperparameter rules suggested in their theory. This gap between theory and application exists in centralized optimization problems and researchers have proposed different methods to mitigate it. 
%However, this gap introduces more serious problems in the decentralized setting for two reasons: (i) the increasing difficulty to estimate the properties of the global loss function for algorithmic design in large-scale problems~\citep{yuan2022decentralized} {\red[why this is increasingly difficult? I think we need to highlight that in distributed problem it is even harder for each individual node to know properties of other functions; in the current writing, the difficulty can come from having a lot of  data, or problem becomes very difficult, so generic metrics such as Lipschitz constants are difficult.   ]}; (ii) the extra error introduced by the heterogeneity of the data distributions on different local nodes for the convergence analysis~\citep{tang2018d, koloskova2020unified}. 
%{\red[shall we mention the unknown network architecture, and the difficulty to compute network related constants such as eigenvalues of the graph laplacian, etc.?]}
However, this gap introduces more serious problems in the decentralized setting for several reasons: (i) In distributed settings, it is hard, if not impossible, for local devices to know the problem information of other devices, even if the network is fully connected~\citep{yuan2022decentralized}; (ii) The network architecture might be largely unknown for algorithmic design, especially when the data are distributed in different physical locations, thus it is difficult to compute network related constants such as eigenvalues of the graph Laplacian; (iii) The extra error introduced by the heterogeneity of the data distributions on each local nodes brings more challenges for convergence analysis~\citep{tang2018d, koloskova2020unified}.

In this work, we close these gaps by designing problem-parameter-free algorithms, i.e., algorithms whose hyperparameters do not require problem information, for decentralized optimization. Specifically, consider the following stochastic decentralized optimization problem:
\begin{equation}\label{eq_decen_problem}
    \min_{x} f(x):=\frac{1}{n}\sum_{i=1}^{n} f_i(x)
\end{equation}
where each $f_i=\E_{\xi_i\sim\mathcal{D}_i}[F_i(x, \xi_i)]$ is stored on a local device/node/agent $i$, which is assumed to be $L_i$-Lipschitz smooth and possibly nonconvex, a standard assumption in the literature. Moreover, we assume that the Lipschitz constant is not available for the algorithmic design. Each local node $i$ is only allowed to access the stochastic function $F_i(x, \xi_i)$ in the algorithm design. Note that for different node $i$ the data distribution could be highly heterogeneous, i.e., each $\xi_i$ follows completely different distributions $\mathcal{D}_i$. Also, each local agent is only connected to a limited amount of neighboring agents, forming an undirected connected graph, which is summarized by the doubly stochastic mixing matrix $W$ (see Section \ref{sec_method}). To be more specific, in this paper, ``problem-parameter-free'' means that the hyperparameters of the algorithm (such as learning rate) do not depend on problems parameters such as $L_i$ and $W$. The goal of this paper is thus to design such algorithms for solving \eqref{eq_decen_problem}.

The most straightforward method for decentralized optimization is decentralized stochastic gradient descent (D-SGD) where each local device runs stochastic gradient descent then communicates the update with their neighbors to form the next iterate. In \citet{lian2017can}, the authors provided a convergence analysis under the assumption of bounded heterogeneity, i.e., the gradient distributions across different devices are similar. To remove this assumption, another famous method is the decentralized gradient tracking algorithm (D-SGT, Algorithm \ref{algo_decen_grad_tracking}, see \citet{xu2015augmented,di2016next,nedic2017achieving,qu2017harnessing,pu2021distributed,koloskova2021improved,liu2023decentralized}) which efficiently guarantees convergence without requiring bounded heterogeneity, also yields superior numerical performances. We thus first inspect the convergence of D-SGT algorithm under the problem-parameter-free setting. Our analysis shows that D-SGT could converge when the hyperparameters are problem-parameter-free. Besides standard assumption of local Lipschitz smooth, however, this convergence result additionally requires the local functions to be {\it Lipschitz continuous} (i.e., having bounded gradient). To remove this restrictive assumption, we propose a new decentralized {\it normalized} averaged stochastic approximate gradient tracking (D-NASA, Algorithm \ref{algo_decen_normalized_averaged_grad_tracking}) which enjoys parameter-free convergence without additional assumptions. 

%, which means that the hyperparameters (i.e. learning rate) of the algorithm do not depend on the problem parameter (such as Lipchitz constant of the objective). 

Our contributions are summarized as follows.
\begin{itemize}
\vspace{-0.2cm}
    \item {\bf New analysis of D-SGT.} We investigate D-SGT (Algorithm \ref{algo_decen_grad_tracking}) and point out that one can use a learning rate that is problem-parameter-free and still guarantee the convergence, at the expense of an additional assumption: local functions are Lipschitz continuous, i.e., bounded local function gradients. This is a rather strong requirement since it implies bounded heterogeneity among different nodes (see Section \ref{sec_convergence_dsgt}). The analysis also indicates that D-SGT can no longer achieve a linear speedup under this setting. %These facts motivate us to propose more efficient problem-parameter-free algorithms.

\vspace{-0.2cm}
    \item {\bf A new parameter-free algorithm.} We propose a fully problem-parameter-free algorithm (D-NASA, Algorithm \ref{algo_decen_normalized_averaged_grad_tracking}) based on certain normalization technique that does not require information of global Lipschitz constant or spectral gap of the topology of the problem. The convergence of D-NASA is guaranteed without any additional assumption. %neither data homogeneity nor global Lipschitz continuous. 
    The convergence result matches the lower bound for nonconvex stochastic optimization and still enjoys the desired linear speedup. %, i.e., a superior performance over the centralized counterparts.

\vspace{-0.2cm}
    \item {\bf Normalization controls consensus error.} D-NASA utilizes a novel control over the consensus error. Specifically, we notice that normalized update efficiently helps the control of the consensus error, and enables controlling the cumulative consensus error directly by stepsizes (see Section \ref{sec_convergence_dnasa}). This opens the door of adapting a wide class of normalization-based adaptive algorithms to the decentralized setting, and its fine-grained analysis is of independent interest. %{\red[not quite understand this paragraph]}\jl{Rewrote this part}

\vspace{-0.2cm}
    \item {\bf Numerical evidences.} We conduct extensive numerical study to verify our findings. We observe linear speedup effect of D-NASA with the stepsize exactly predicted by our theory. We also show that D-NASA compares favorably with existing algorithms D-SGD, D-SGT and D-ASAGT in terms of convergence speed. We empirically demonstrate that D-NASA does not require any parameter tuning for a wide range of Lipschitz smooth parameters, and network topology. %{\color{red}[shall we say for a wide range of Lispchitz smooth parameter, and network topology?][how about network topology?]}\jl{we actually tested for two graphs: random and ring graph. Not sure if it's worth mentioning here...}{\red[for randomly generated graph, the diameter/spectrum radius should be different each time, so we should be ok?][but worth having more topologies tested (appendix?), since this is a strong gain for our benefit. ]} \jl{For random graph, I set the probability of each two node connected to be high (0.8) to make sure the graph is connected.} \jl{Added more experiments in the appendix} 
    Without this technique, the stepsize tuning process can be time-consuming since the optimal choices of the hyperparameters vary drastically when datasets change.
\end{itemize}

%{\red[summarize the rest of the two paragraphs using a couple of words at the beginning.]}

\noindent\textbf{Notation.} We denote $\X^t:=[x_1^{t},..., x_n^{t}]$ which is the collection of local variables $x_i^t$ at iteration $t$ for $i=1,...,n$ as column vectors. $\Bar{x}^t:=\frac{1}{n}\sum_{i=1}^{n}x_i^t$ is the average of all local variables. The same convention applies to $\U$, $\V$, $\Z$ and $\Bar{u}^t$, $\Bar{v}^t$, $\Bar{z}^t$. Also denote by $\Bar{\X}^t=\Bar{x}^t\ones^\top=\frac{1}{n}\X^t\ones\ones^\top$ the collection of average of local variables, where $\ones$ is $n$-dimensional all one column vector. The same convention applies to $\bar{\U}$, $\bar{\V}$ and $\bar{\Z}$. We use $\|\cdot\|$ to represent the Euclidean vector norm and matrix Frobenius norm to simplify the notation. For matrix 2 norm (i.e., spectral norm) we use $\|\cdot\|_{2}$.

\subsection{Related works}

\begin{table*}[t!]
\caption{Comparison of D-NASA (Algorithm \ref{algo_decen_normalized_averaged_grad_tracking}) with some widely-used decentralized stochastic nonconvex optimization algorithms: D-SGD~\citep{lian2017can}, D$^2$~\citep{tang2018d} and D-SGT~\citep{koloskova2021improved}. `Other aspt' refers to the additional assumptions required for theoretical convergence (Note that the parameters in the assumptions might not be available to the algorithm), where `Hetero' stands for bounded heterogeneity, and all algorithms require the stochastic bounded variance and the deterministic gradients begin Lipschitz continuous; `Info Required' refers to the problem parameters that the algorithm parameters (such as stepsizes) should depend on to achieve the sample complexity, where ``smoothness" is the Lipschitz constant of the global gradient, and ``variance" is the variance of the stochastic oracle; All algorithms in this table require $\mathcal{O}(n^{-1}\epsilon^{-4})$ oracle calls to achieve an $\epsilon$-stationary point.}
\label{table: comparison}
\begin{center}
\begin{sc}
\begin{tabular}{ccc}
\toprule
\textbf{Algorithm} & \textbf{Other Aspt}  & \textbf{Info Required} \\ %& \textbf{Complexity}\\ 
\toprule
% \textbf{NEXT}* &  L-cont & NA & Asymptotic  \\
\textbf{D-SGD} &  Hetero  & Smoothness, variance \\ %&  $\mathcal{O}(n^{-1}\epsilon^{-4})$ \\ 
\textbf{D$^2$} &  None  & Smoothness, net-topology \\ %&  $\mathcal{O}(n^{-1}\epsilon^{-4})$ \\ 
\textbf{D-SGT} &  None  & Smoothness, net-topology  \\ %&  $\mathcal{O}(n^{-1}\epsilon^{-4})$ \\
\textbf{D-NASA} (ours) & None  & None  \\ %& $\tilde{\mathcal{O}}(n^{-1}\epsilon^{-4})$ \\ 
\bottomrule
\end{tabular}
\end{sc}

%{\red[change NA to 'none'?][smoothness means the smoothness of the entire function? what's the variance here means?]}\jl{I further explained the meaning of smoothness and variance in the caption above}

% \footnotesize{*: We do not compare other popular methods such as NEXT~\citep{di2016next}, which only proves asymptotic convergence under deterministic setting so we do not do the comparison.}

% \footnotesize{$\dagger$: In particular, D-SGD in \cite{lian2017can} requires the information of Lipschitz smoothness parameter and variance of stochastic gradients. D-SGT in \cite{koloskova2021improved} requires the Lipschitz smoothness parameter and $\lambda_2, \lambda_n$ (see Section \ref{sec_method}), which we summarize as 'net-topology'. Our D-NASA (Algorithm \ref{algo_decen_normalized_averaged_grad_tracking}) does not require any problem information to select the algorithm parameters.}

% 'Complexity' represents the number of oracles (stochastic gradients) needed on each agent to obtain an $\epsilon$-stationary point, where $n$ is the number of total agents, and the factor $n^{-1}$ indicates a linear speedup effect.
\end{center}
\end{table*}

\noindent\textbf{Decentralized optimization}
While the study of decentralized optimization algorithms has a long history~\citep{tsitsiklis1984problems, ram2009asynchronous, yan2012distributed, yuan2016convergence}, their distinctive advantages, such as robustness, scalability and privacy preserving, in comparison to centralized setting like~\citet{li2014scaling}, were not well understood both theoretically and empirically until the case study conducted by \citet{lian2017can}. Despite its great success in characterizing the superiority of decentralized training over the centralized setting, the analysis therein replies on a bounded gradient heterogeneity assumption, which was later removed by follow-up works such as D$^2$~\citep{tang2018d}.

Motivated by the empirical success of decentralized training, another line of work focused on improving the convergence rates of decentralized algorithms. Vanilla decentralized gradient descent with a fixed stepsize is known to only converge to a neighborhood of the optimal solution even under the deterministic and strongly convex setting~\citep{yuan2016convergence}. One important technique to mitigate this effect is gradient tracking, which was introduced in control community~\citep{xu2015augmented, di2016next, nedic2017achieving, qu2017harnessing} to improve the convergence rate in the deterministic setting. Later this method was revealed to be helpful to remove the bounded gradient heterogeneity assumption~\citep{zhang2019decentralized, lu2019gnsd, pu2021distributed, koloskova2021improved} in convergence analysis.
A more recent technique of moving-average updates (momentum) have been studied in both decentralized optimization and federated learning setting~\citep{pmlr_v216_xiao23a, cheng2023momentum} to further improve the rate of convergence.

It is worth noticing that the above works all require knowledge about the global problem to design their algorithms. Under the assumption that the local functions are Lipschitz continuous, NEXT~\citep{di2016next} is able to achieve a problem-parameter-free asymptotic convergence (in \textbf{deterministic} setting). We point out again that Lipschitz continuity of the objective functions is a strong assumption that implies boundedness of gradients and bounded heterogeneity (see Section \ref{sec_convergence_dsgt}). % and thus NEXT falls in the same category as our problem-parameter-free analysis for D-SGT in Section \ref{sec_convergence_dsgt}.

Other interesting research topics in decentralized optimization include network topology~\citep{neglia2020decentralized, koloskova2020unified}, communication compression~\citep{tang2018communication, koloskova2019decentralized}, large-model training~\citep{gan2021bagua, yuan2022decentralized}, adaptive algorithms~\citep{chen2023convergence}, to name a few.

% As compared to centralized algorithms, decentralized algorithms are provably preferable in large-scale machine learning problems under certain circumstances due to its robustness, scalability and privacy preserving~\citep{lian2017can}. 

\noindent\textbf{Parameter-free optimization} (Problem-) Parameter-free optimization refers to the algorithms that require no/few information needed from the problem so that the algorithm converges without any tedious process of hyperparameter-tuning. For deterministic smooth optimization, one could show the convergence of gradient descent to either the optimal (convex) or the stationary point (nonconvex) when the stepsize $\eta$ is smaller than $2/L$, where $L$ is the Lipschitz smooth constant~\citep{nesterov2018lectures}. When problem parameters such as $L$ are not available, one usually uses backtracking line-search to determine the stepsize. Recently, there is a line of research initiated by \citet{malitsky2019adaptive} that adaptively estimates the local curvature information in each iteration and does not require the knowledge of $L$. See~\citet{malitsky2023adaptive,  latafat2023convergence, latafat2023adaptive,li2023simple, zhou2024adabb} for more recent works on this subject. Currently, these adaptive methods are for deterministic problems and it remains an interesting direction to extend them to stochastic and decentralized settings.

%{\red[do we need to cite the counterpart of our algorithm in the centralized setting, and provide discussion here?]}\jl{these works are cited in the paragraph below. I also added a short discussion at the end of section 2}

For stochastic gradient descent for solving convex problems, the current convergence result requires either a constant step upper bounded by $1/(2L)$, or a diminishing stepsize $\eta_t=\eta/\sqrt{t}$ with $\eta$ still upper bounded by terms related to $L$~\citep{garrigos2023handbook}. Sufficiently small stepsize guarantees the convergence since the analysis resembles the gradient flow regime, yet this is usually inconsistent with empirical studies, which encourage the stepsize to be large as long as there is no divergence. The stepsize can be chosen up to $10^3$ and $10^4$ in some logistic regression problems (see Section C.2 in \citet{grazzi2020iteration}), which indicates that optimal choices of stepsizes in SGD heavily depend on problem parameters. 

For nonconvex stochastic optimization, various adaptive methods, such as AdaGrad~\citep{duchi2011adaptive,mcmahan2010adaptive}, AMSGrad-Norm~\cite{reddi2019convergence}, NSGD-M~\cite{cutkosky2020momentum}, are proved to be convergent without any knowledge of the parameters~\cite{faw2022power,junchi2023two,hubler2023parameter}, which are thus believed to be more robust algorithms comparing to SGD. Another line of works for the stochastic/online convex optimization is to use the accumulative norm of the stochastic gradient to design adaptive stepsizes~\citep{carmon2022making,ivgi2023dog}. These research results emphasize the optimal dependency on $\|x^0-x^*\|$, i.e., the distance from the initial to the optimal point, and it is not clear how these works adapt to the nonconvex problems. 

Parameter-free stochastic optimization in decentralized setting is unexplored. It is natural to ask whether one can achieve parameter-free decentralized training, given the unique challenges such as communication complexity and heterogeneous data distribution across agents. We provide an affirmative answer in this paper, and in Table \ref{table: comparison} we make the comparison between our Algorithm \ref{algo_decen_normalized_averaged_grad_tracking} and existing well-known algorithms: D-SGD~\citep{lian2017can}, D$^2$~\citep{tang2018d} and D-SGT~\citep{koloskova2021improved}\footnote{We do not compare with NEXT~\citep{di2016next}, which only proves asymptotic convergence under deterministic setting.}. In particular, D-SGD in \cite{lian2017can} requires the information of Lipschitz smoothness parameter and variance of stochastic gradients. D-SGT in \cite{koloskova2021improved} requires the Lipschitz smoothness parameter and $\lambda_2, \lambda_n$ (see Section \ref{sec_method}), which we summarize as `net-topology'. Our D-NASA (Algorithm \ref{algo_decen_normalized_averaged_grad_tracking}) does not require any problem information to select the algorithm parameters.

\section{Methodology}\label{sec_method}
We now present the full methodology of our algorithm. First we recall the decentralized communication topology with a weighted undirected graph $(V,W)$. The vertex set $V=\{1,2,...,n\}$ is the set of local device/nodes, and $W = (W_{i,j})\in\RR^{n\times n}$ is a symmetric doubly stochastic matrix known as weighted adjacency matrix, i.e., $W$ satisfies the following properties: (1) $W_{i, j}\in [0, 1]$, $\forall i,j$; (2) $W_{i, j} = W_{j, i}$, $\forall i, j$, i.e., $W^\top=W$; and (3) $\sum_{j=1}^{n}W_{i, j}=1$, $\forall i$, i.e., $W\ones=\ones$ and $\ones^\top W=\ones^\top$. Intuitively, $W_{i, j}$ represents how well the communication between node $i$ and $j$ is, and $W_{i, j}=0$ if and only if $i$ and $j$ are not communicating. Note that we assume that the eigenvalues of $W$ satisfy $1=\lambda_1>\lambda_2\geq\cdots\geq\lambda_n>-1$, and 
\begin{equation}\label{eq_rho}
    \rho:=\max\{|\lambda_2|, |\lambda_n|\}< 1,
\end{equation}
which is standard in decentralized optimization literature~\citep{lian2017can, tang2018d}. This ensures the communication graph is strongly connected, and after each round of communication with neighbors, the consensus error (i.e., $\sum_{i=1}^{n}\norm{a_i - \bar a}^2$ where $a_i$ is a vector owned by the $i$-th agent only) decreases at a controllable rate. We assume $W$ satisfies the above properties throughout the paper and thus will not explicitly state them in the theorems.

% For network topology we have the following assumption.
% \begin{assumption}\label{assump_network_topology}
%     Suppose the weighted adjacency matrix $W=(W_{i,j})\in \RR^{n\times n}$ associated with the communication network is symmetric doubly stochastic, i.e., $W^\top = W,\ W\mathbf{1} = \mathbf{1}$. Denote by $\lambda_n\leq ...\leq \lambda_1$ all eigenvalues of $W$, and $\rho:=\max\{|\lambda_2|, |\lambda_n|\} < 1$.
% \end{assumption}

Now we recall the decentralized stochastic gradient tracking (D-SGT)~\citep{zhang2019decentralized, lu2019gnsd, pu2021distributed, koloskova2021improved} in Algorithm \ref{algo_decen_grad_tracking}. The algorithm takes a gradient step at each local node, keeps a tracker $u_i^{t}$ to approximate the global stochastic gradient, and executes a communication round in each iteration to achieve consensus among agents. A simple arithmetic verification shows that $\Bar{u}^t=\Bar{v}^t$ for all iteration number $t>0$. This key mechanism guarantees that the averaged gradient tracker $\Bar{u}^t$ is close to full gradient $\nabla f$ provided the consensus error $\sum_{i=1}^{n}\norm{x_i^t - \bar x^t}^2$ is small. % and as the consensus error $\sum_{i=1}^{n}\norm{u_i^t - \bar u^t}^2$ decreases, each local tracker $u_i^t$ is a good approximation of the global gradient $\nabla f$.
However, as we will show in Section \ref{sec_convergence}, this popular D-SGT algorithm cannot achieve a problem-parameter-free convergence with linear speedup even when we assume that the local functions are Lipschitz continuous, i.e., their gradients are bounded. We primarily use D-SGT to showcase the difficulties of applying these algorithms to modern machine learning applications, as we essentially still need to tune the algorithm parameters for a better performance in distributed training, which is largely impossible~\cite{yuan2022decentralized}.

\begin{algorithm}[!ht]
\caption{Decentralized stochastic gradient tracking (D-SGT)}\label{algo_decen_grad_tracking}
\begin{algorithmic}[1]
    \STATE {\bfseries Input:} $T$, $\{\eta_t\}$, $u_i^0=v_i^{0}=\nabla F_i(x_i^0,\xi_i^{0})$
    \STATE {\bfseries Output:} $\Tilde{x}=x^{T}$ or uniformly from $\{x^1,...,x^T\}$
    \FOR{$t=0,...,T-1$}
        \FOR{each node $i=1,...,n$ (in parallel)}
            \STATE $x_i^{t+1} \leftarrow \sum_{j=1}^{n} W_{i,j}(x_j^t - \eta_t u_j^t),$ \\
            \STATE $ v_i^{t+1}\leftarrow \nabla F_i(x_i^{t+1},\xi_i^{t+1})$ \\
            \STATE $u_i^{t+1} \leftarrow \sum_{j=1}^{n} W_{i,j} u_j^{t} + v_i^{t+1} - v_i^{t}$ \\
        \ENDFOR
    \ENDFOR
\end{algorithmic}
\end{algorithm}

To overcome this obstacle, we propose decentralized normalized averaged stochastic approximation (D-NASA) as in Algorithm \ref{algo_decen_normalized_averaged_grad_tracking}, where we maintain $z_i^t$ as a moving-average update of the tracker $u_i^t$ and then utilize the normalized direction $z_i^t / \norm{z_i^t}$ to update $x_i^t$. Another difference of D-NASA is that we update all the local operations and communicate at the end, simply for the ease of analysis. The moving-average technique, also known as momentum method, was recently introduced to distributed optimization and proven to mitigate client drift in federated learning~\citep{cheng2023momentum} and achieve linear speedup in decentralized composite optimization~\citep{pmlr_v216_xiao23a}. In Section \ref{sec_convergence}, our theory reveals that the normalized direction coupled with the moving-average provably achieves parameter-free decentralized optimization. The combination of normalization and moving-average was explored in \citet{cutkosky2020momentum,hubler2023parameter}, yet it is unclear and highly non-trivial to understand if one can achieve parameter-free convergence when each of the local node is normalized only by its local norm of gradients.

\begin{algorithm}[!ht]
 \caption{Decentralized normalized averaged stochastic approximation (D-NASA)}
 \label{algo_decen_normalized_averaged_grad_tracking}
\begin{algorithmic}[1]
    \STATE {\bfseries Input: } $T$, $\{\eta_t\}$, $\{\alpha_t\}$, $x_i^0=z_i^0=v_i^{0}=0$
    \STATE {\bfseries Output: } $\Tilde{x}=x^{T}$ or uniformly from $\{x^1,...,x^T\}$
    \FOR{$t=0,...,T-1$}
        \FOR{each node $i=1,...,n$ (in parallel)}
            \STATE $ \Tilde{x}_i^{t+1} \leftarrow x_i^t - {\frac{\eta_t}{\norm{z_i^t}}} z_i^t$ \\
            \STATE $ v_i^{t+1}\leftarrow \nabla F_i(x_i^{t},\xi_i^{t})$\\
            \STATE $\Tilde{u}_i^{t+1} \leftarrow u_i^{t} + v_i^{t+1} - v_i^{t}$ \\
            \STATE $\Tilde{z}_i^{t+1} \leftarrow (1-\alpha_t) z_i^t + \alpha_t \Tilde{u}_i^{t+1}$
        \ENDFOR
        
        \texttt{\# Communication} \\
        \STATE $[x_1^{t+1},..., x_n^{t+1}]\leftarrow [\Tilde{x}_1^{t+1},..., \Tilde{x}_n^{t+1}] W$ \\
        \STATE $[u_1^{t+1},..., u_n^{t+1}]\leftarrow [\Tilde{u}_1^{t+1},..., \Tilde{u}_n^{t+1}] W$ \\
        \STATE $[z_1^{t+1},..., z_n^{t+1}]\leftarrow [\Tilde{z}_1^{t+1},..., \Tilde{z}_n^{t+1}] W$ \\
    \ENDFOR
\end{algorithmic}
\end{algorithm}

\section{Convergence analysis}\label{sec_convergence}
In this section we analyze the convergence properties of our algorithms. We have the following standard assumptions for our theoretical analysis of Algorithm \ref{algo_decen_grad_tracking} and \ref{algo_decen_normalized_averaged_grad_tracking}.

\begin{assumption}\label{assump_l_smooth}
    The function $f_i$ is $L_i$-Lipschitz smooth, i.e.
    $$
    \|\nabla f_i(x) - \nabla f_i(y)\| \leq L_i\|x - y\|.
    $$
    As a result, $f$ is $L$-Lipschitz smooth with $L=\frac{1}{n}\sum_i L_i$.
\end{assumption}

Next, we also have the following standard assumption on the mean and variance of each local gradient estimator. Denote the filtration generated by the random variables sampled upon the $t$-th iteration as $\mathcal{F}_t$, i.e. $\mathcal{F}_0=\{\emptyset,\Omega\}$ and 
$$
    \mathcal{F}_t:=\sigma(\xi_i^k|i=1,...,n,\ k=0,...,t), \ \forall t\geq 1
$$
where $\sigma$ is the $\sigma$-algebra generated by the random variables.
\begin{assumption}\label{assump_bdd_variance}
    The stochastic gradient estimator is unbiased and with bounded variance, i.e.,
    \begin{align*}
        &\E_{\xi_i}[\nabla F_i(x, \xi_i)] = \nabla f_i(x),\\ &\E_{\xi_i}\|\nabla F_i(x, \xi_i) - \nabla f_i(x)\|^2\leq \sigma^2.
    \end{align*}
    Moreover, we assume $\{\xi_i^{t+1}:i=1,...,n\}$ are independent given $\mathcal F_t$. % so that we can take the expectation of $\xi_i^{t+1}$ independently.
\end{assumption}
Note that Assumption \ref{assump_bdd_variance} is only imposed on each local stochastic function, and does not imply any bound for the difference between local and global functions. We now define the notion of stationarity for this paper.
\begin{definition}
    For any $\epsilon>0$, we say an algorithm finds an $\epsilon$-stationary point, if an output sequence $\{\bar x^{t}\}_{t=0}^{T}$ generated by the algorithm satisfies 
    \begin{align*}
        \frac{1}{T}\sum_{t=0}^{T-1}\E\|\nabla f(\Bar{x}^t)\|\leq\epsilon.
    \end{align*}
\end{definition}
%\jl{added the notion of linear speedup here:}
We say that an algorithm achieves \textbf{linear speedup} if it takes $T\propto n^{-1}$ oracles calls at each node to achieve an $\epsilon$-stationary point. % {\red[let's discuss. I think this is a bit too specific? e.g., why $\epsilon^{-4}$?]} \jl{$\epsilon^{-4}$ is the lower bound for stochastic nonconvex optimization based on Carmon's work. But I think here maybe we just need to mention a $n^{-1}$ dependency? I've modified correspondingly}

\subsection{Parameter-free convergence theory for D-SGT}\label{sec_convergence_dsgt}
We first show the convergence analysis of the D-SGT algorithm in which the learning rate does not depend on problem parameters. However, this convergence result requires the following Lipschitz continuity assumption on functions $f_i$.
\begin{assumption}\label{assump_l_continuous}
    The function $f_i$ is $G_i$-Lipschitz continuous, i.e.
    $$
    \|f_i(x) - f_i(y)\| \leq G_i\|x - y\|.
    $$
    As a result, $f$ is $G$-Lipschitz continuous with $G=\frac{1}{n}\sum_i G_i$.
\end{assumption}
Note that Assumption \ref{assump_l_continuous} is only used in the parameter-free convergence analysis for the D-SGT algorithm (Algorithm \ref{algo_decen_grad_tracking}). This is a very strong assumption since in convex optimization, Lipschitz continuity of the objective functions (or bounded subgradient) can readily give a parameter-independent convergence result by taking the stepsize to be $\mathcal{O}(1/\sqrt{t})$~\citep{boyd2003subgradient}. Moreover, it implies that each function $f_i$ has bounded gradients, i.e., $\norm{\nabla f_i(x)}\leq G_i$, which further indicates the bounded heterogeneity condition since
$\norm{\nabla f_i(x) - \nabla f(x)}\leq \norm{\nabla f_i(x)} + \frac{1}{n}\sum_{j=1}^{n}\norm{\nabla f_j(x)}\leq G_i + G.$

We point out that, even under such a strong assumption, we are not able to show a linear speed up effect for D-SGT. Specifically, we have the following theorem.
\begin{theorem}\label{thm_fix_step_dsgt}
    Suppose Assumptions \ref{assump_l_smooth}, \ref{assump_bdd_variance} and \ref{assump_l_continuous} hold, also take $\eta_t=\eta T^{-1/2}$ (constant) or $\eta_t=\eta t^{-1/2}$ (diminishing, $\eta_0=0$ for this case) for $\eta>0$, the update of Algorithm \ref{algo_decen_grad_tracking} satisfies:
    \begin{align*}
        &\frac{1}{T} \sum_{t=0}^{T-1}\E\|\nabla f(\bar{x}^t)\|^2\\& \leq \tilde{\mathcal{O}}\bigg( \frac{\Delta_0/\eta + (L\sigma^2/n+L G^2)\eta}{\sqrt{T}} + \frac{\Tilde{\rho} L^2 \eta^2}{T}(\sigma^2+G^2) \bigg).
    \end{align*}
    Here $\Tilde{\rho}>0$ is a parameter dependent on $\rho$ in \eqref{eq_rho}, $\Delta_0=f(\Bar{x}^0)-f^*$ is the initial function value gap and we omit higher-order and logarithmic terms in $\tilde{\mathcal{O}}$.
\end{theorem}
\begin{remark}
    The rate of $\mathcal{O}(1/\sqrt{T})$ matches the lower bound for nonconvex stochastic optimization~\cite{arjevani2023lower}, yet it is worth noticing that we are not able to choose the parameter $\eta$ to achieve a linear speedup effect (even if we have access to $n$, the number of nodes), due to the term related to $G$. We also remind the reader that if we assume the access of Lipschitz smooth constant $L$, one can achieve linear speedup for D-SGT as in~\citet{zhang2019decentralized, xin2021improved, koloskova2021improved}. This motivates the design of new algorithms that can achieve linear speedup under problem-parameter-free setting for decentralized optimization, without the restrictive Assumption \ref{assump_l_continuous}.
\end{remark}

\subsection{Parameter-free convergence theory for D-NASA}\label{sec_convergence_dnasa}
Now we analyze the convergence of D-NASA (Algorithm \ref{algo_decen_normalized_averaged_grad_tracking}). Similar to the result for D-SGT, we provide both the result for fixed and diminishing stepsizes. Our analysis depends on a key observation over the control of the consecutive consensus error. By the update of Algorithm \ref{algo_decen_normalized_averaged_grad_tracking} one can get:
\begin{align*}
    &\|\X^{t+1} - \Bar{\X}^{t+1}\|^2\\
    &\leq \frac{1+\rho}{2}\|\X^{t} - \Bar{\X}^{t}\|^2 + \eta_t^2 \frac{1+\rho^2}{1-\rho^2}\|\hat{\Z}^{t} - \Bar{\hat{\Z}}^t\|^2
\end{align*}
where $\hat{\Z}^t:=\left[\frac{z_1^t}{\|z_1^t\|},...,\frac{z_n^t}{\|z_n^t\|}\right]$ is the collection of column vectors of normalized $z_i^t$. Now the key observation is that the consensus error of $\hat{\Z}^t$ is always bounded: 
\begin{align*}
    \|\hat{\Z}^{t} - \Bar{\hat{\Z}}^t\|^2 =\sum_{i=1}^{n}\left\| \frac{z_i^t}{\|z_i^t\|} - \frac{1}{n}\sum_{i=1}^{n}\frac{z_i^t}{\|z_i^t\|} \right\|^2 \leq n.
\end{align*}
Therefore the consecutive consensus error for $\X$ becomes:
\begin{align*}
    \frac{1}{n}\|\X^{t+1} - \Bar{\X}^{t+1}\|^2
    \leq \frac{1+\rho}{2}\frac{1}{n}\|\X^{t} - \Bar{\X}^{t}\|^2 +  \frac{1+\rho^2}{1-\rho^2} \eta_t^2
\end{align*}
and the cumulative consensus error is controlled directly by our stepsize choice $\eta_t$. This indicates that a careful stepsize choice will result in a bounded consensus error, regardless of any problem parameter. Now we state the convergence result for a fixed stepsize as follows.  
\begin{theorem}\label{thm_fix_step}
    Suppose Assumptions \ref{assump_l_smooth} and \ref{assump_bdd_variance} hold and we take $\alpha_t = \sqrt{n/T}$ and $\eta_t=n^{1/4}/ T^{3/4}$ in Algorithm \ref{algo_decen_normalized_averaged_grad_tracking}. The following bounds hold:
    \begin{align*}
        \frac{1}{T} \sum_{t=0}^{T-1}&\E\|\nabla f(\bar{x}^t)\| \leq \mathcal{O}\bigg( \frac{\Delta_0 + L + \sigma}{n^{1/4}T^{1/4}} + \frac{\Tilde{\rho}^2(\sigma+L)n^{1/2}}{T^{1/2}} \bigg), \\
        \frac{1}{T} \sum_{t=0}^{T-1}&\E\|\Bar{z}^t - \nabla f(\Bar{x}^t)\| \leq \mathcal{O}\bigg( \frac{L+\sigma}{n^{1/4} T^{1/4}} + L\Tilde{\rho}\frac{n^{1/4}}{T^{1/2}} \bigg), \\
        \frac{1}{T} \sum_{t=0}^{T-1}&\frac{1}{n}\left[\|\X^{t} - \Bar{\X}^{t}\|^2 + \|\Z^{t} - \Bar{\Z}^{t}\|^2\right] \\
        &\leq \mathcal{O}\bigg( \Tilde{\rho}\frac{n^{1/4}}{T^{1/2}}+ \Tilde{\rho}^2( \sigma^2+ L^2)\frac{n^2}{T} \bigg).
    \end{align*}
    Here $\Tilde{\rho}>0$ is a parameter dependent on $\rho$ in \eqref{eq_rho}, $\Delta_0=f(\Bar{x}^0)-f^*$ is the initial function value gap and we omit higher-order terms in $\mathcal{O}$. Note that the above three bounds correspond to stationarity, approximation to gradient and consensus errors.
\end{theorem}

\begin{remark}\label{rmk_fix_step_linear_speedup}
    To make $1/T\sum_{t=0}^{T-1}\E\|\nabla f(\Bar{x}^t)\|\leq\epsilon$, we need $T=\tilde{\mathcal{O}}(1/(n\epsilon^4))$, which matches the lower bounds as in~\citet{lu2021optimal, arjevani2023lower}, and also indicates the linear speedup effect \citep{lian2017can}. Note that the approximation error $\|\Bar{z}^t-\nabla f(\Bar{x}^t)\|$ also enjoys linear speedup effect. Readers might realize that this choice of parameter requires prior knowledge of the total number of nodes. We presume that it is impossible to achieve linear speedup if we are using none of the problem information. Moreover, this choice of algorithm parameters still does not require global information about the loss function or the topological information about the communication graph, thus it is better than existing algorithms in the literature in decentralized optimization, as we have presented in Table \ref{table: comparison}.
\end{remark}

To free the algorithm parameters even from the total number of iterations $T$, we also present the result when we do not fix the total number of iterations in advance and the stepsize will be diminishing in Theorem \ref{thm_dim_step}. 
\begin{theorem}\label{thm_dim_step}
    Suppose Assumptions \ref{assump_l_smooth} and \ref{assump_bdd_variance} hold, also take $\alpha_t = \sqrt{n/t}$ and $\eta_t=n^{1/4}/ t^{3/4}$ for any $t$ (take $\eta_0=\alpha_0=0$), the update of Algorithm \ref{algo_decen_normalized_averaged_grad_tracking} satisfies: 
    \begin{equation*}
    \begin{split}
        \frac{1}{T} \sum_{t=0}^{T-1}&\E\|\nabla f(\bar{x}^t)\| \leq \Tilde{\mathcal{O}}\bigg(\frac{\Delta_0+L+\sigma}{n^{1/4}T^{1/4}}\\
        &+\frac{ L\Tilde{\rho}n^{1/4} + \Tilde{\rho}^2(\sigma+L)n^{1/2}+L\Tilde{\rho}^3 n^{3/4}}{T^{1/4}}\bigg), \\
        \frac{1}{T} \sum_{t=0}^{T-1}&\E\|\Bar{z}^t - \nabla f(\Bar{x}^t)\| \leq \Tilde{\mathcal{O}}\bigg( \frac{L+\sigma+L\Tilde{\rho}n^{1/2}}{n^{1/4} T^{1/4}} \bigg), \\
        \frac{1}{T} \sum_{t=0}^{T-1}&\frac{1}{n}\E\left[\|\X^{t} - \Bar{\X}^{t}\|^2 + \|\Z^{t} - \Bar{\Z}^{t}\|^2\right] \\
        &\leq \mathcal{O}\bigg(\frac{\Tilde{\rho}^2(\sigma^2+ L^2 + \Tilde{\rho} n^{1/2})}{T}\bigg),
    \end{split}
    \end{equation*}
    where $\Tilde{\rho}$ and $\Delta_0$ are the same as Theorem \ref{thm_fix_step} and we omit logarithmic factors in $\Tilde{\mathcal{O}}$.
\end{theorem}

\begin{remark}\label{rmk_dimishing_step}
  To ensure  $1/T\sum_{t=0}^{T-1}\E\|\nabla f(\Bar{x}^t)\|\leq\epsilon$, we need $T=\tilde{\mathcal{O}}(1/\epsilon^4)$, which again matches the lower bound as in~\cite{lu2021optimal, arjevani2023lower} up to logarithmic factors. Yet we are not able to achieve a concrete linear speedup effect with this choice of algorithm parameters. This might root back to our estimation of certain error terms in the proof (see Lemma \ref{lemma_consensus}). Nevertheless, we show in the numerical experiments (see Figure \ref{fig:linear_speedup_plots}) that the stepsize choices $\alpha=\sqrt{n}$ and $\eta=n^{1/4}$ can still achieve linear speedup empirically, and we thus stick to this choice of parameters in experiments.
\end{remark}

\section{Numerical experiments}\label{sec_experiments}
In this section, we test D-NASA (Algorithm \ref{algo_decen_normalized_averaged_grad_tracking}) numerically and compare it with existing algorithms such as D-SGD~\cite{lian2017can}, D-SGT (Algorithm \ref{algo_decen_grad_tracking}) and D-ASAGT~\cite{pmlr_v216_xiao23a}\footnote{\citet{pmlr_v216_xiao23a} considers nonsmooth proximal version of the algorithm. In our numerical experiments we simply regard the nonsmooth proximal term as zero.}. We follow the experimental setup in the code framework of \citet{mancino2023proximal} to test the algorithms on real datasets using mpi4py~\citep{dalcin2021mpi4py} and PyTorch~\citep{paszke2019pytorch}.

\subsection{Synthetic data experiments}
We first use synthetic data to verify the linear speedup effect of D-NASA (Algorithm \ref{algo_decen_normalized_averaged_grad_tracking}). We consider a simple linear regression model where the data sample at each node $\xi=(X, Y)$ is generated by $Y=X^\top\theta_\star +\epsilon$ where $X, \theta_\star\in\RR^d$ and $\epsilon\sim\mathcal{N}(0, \sigma^2)$ are Gaussian noise. We solve the following least-square problem:
\begin{align}\label{eq_synthetic_ls}
\min_{\theta \in \RR^d} \frac{1}{n} \sum_{i=1}^n \mathbb{E}_{(X, Y) \sim \mathcal{D}_i}[(Y-X^{\top} \theta)^2].
\end{align}
In our experiment, we set $d=100$, data $X\sim\mathcal{N}(0, I_d)$ and $\sigma=0.1$. We simulate streaming data samples with batch
size = $1$ for training and $10000$ data samples per node for evaluations. We employ a ring topology for the network where
self-weighting and neighbor weights are set to be $1/3$. For D-NASA, we try both fixed stepsizes ($\alpha_t=\sqrt{n/T}$, $\eta_t=n^{1/4}/T^{3/4}$) and diminishing stepsizes ($\alpha_t=\sqrt{n/t}$, $\eta_t=n^{1/4}/t^{3/4}$) where the total number of iteration $T=15000$ and $n\in\{5, 10, 20\}$. Figure \ref{fig:linear_speedup_plots} shows results of our experiment. It could be seen that with more number of nodes D-NASA is more efficient in terms of both test loss and the norm of the gradient (at the global point $\bar{x}^t$). 

\begin{figure*}[!h]
    \begin{center}
    \subfigure[]{\includegraphics[width=0.44\textwidth]{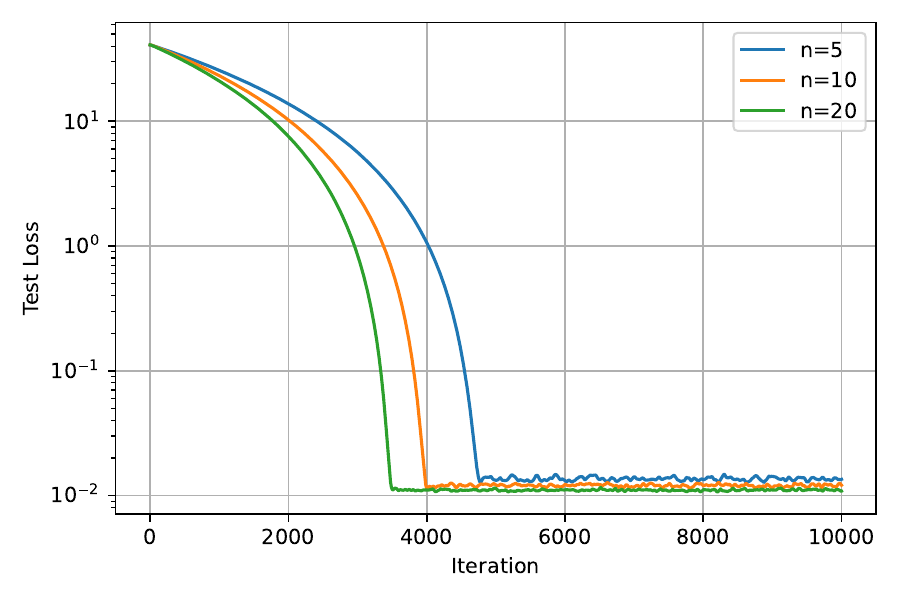}}
    \subfigure[]{\includegraphics[width=0.44\textwidth]{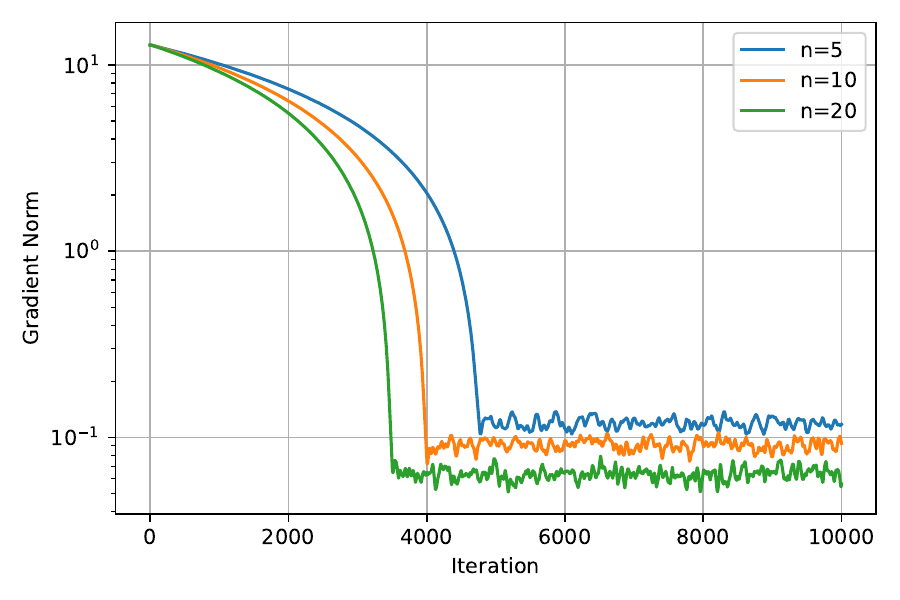}}
    \subfigure[]{\includegraphics[width=0.44\textwidth]{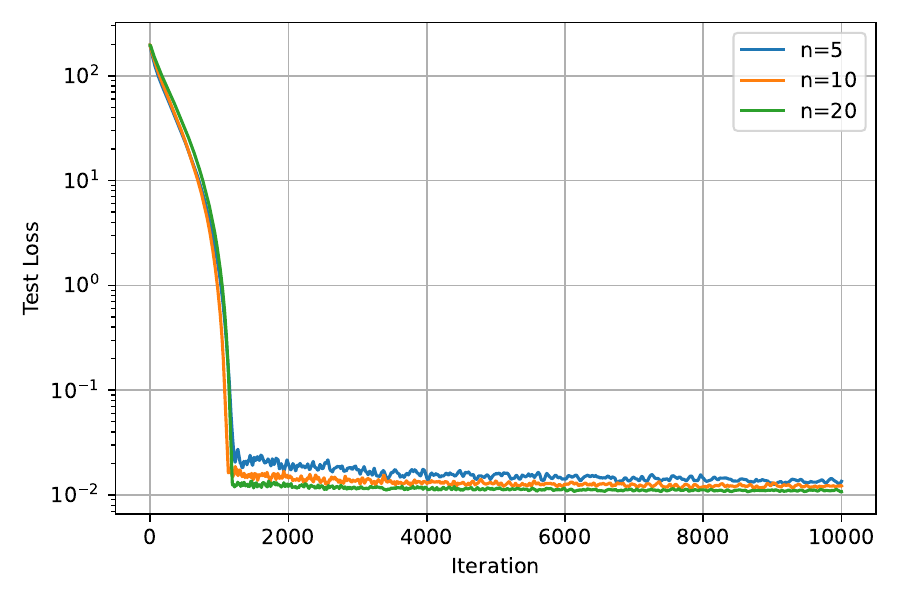}}
    \subfigure[]{\includegraphics[width=0.44\textwidth]{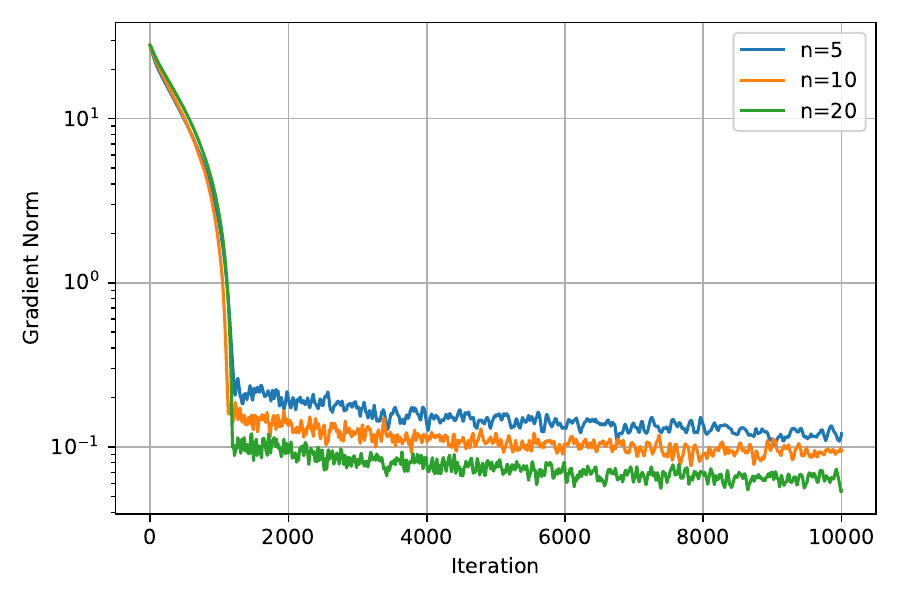}}
    
    \caption{The convergence curve of Algorithm \ref{algo_decen_normalized_averaged_grad_tracking} to solve \eqref{eq_synthetic_ls} with different choice of number of nodes/devices $n\in\{5, 10, 20\}$. The above two figures correspond to fixed stepsizes ($\alpha_t=\sqrt{n/T}$, $\eta_t=n^{1/4}/T^{3/4}$) and below two corresponds to diminishing stepsizes ($\alpha_t=\sqrt{n/t}$, $\eta_t=n^{1/4}/t^{3/4}$), respectively.}
    \label{fig:linear_speedup_plots}
    \end{center}
\end{figure*}

Next, we compare D-NASA on \eqref{eq_synthetic_ls} with the other three algorithms. We still set $d=100$, yet with a spike model with $X\sim\mathcal{N}(0, \diag(100, 1, \cdots, 1))$ where only the first entry has a large variance, in order to make the Lipschitz smooth constant of \eqref{eq_synthetic_ls} large. It is worth noting that despite the fact that conservative constant stepsize choices (usually $\mathcal{O}(\sqrt{n/T})$) can lead to linear speedup effect in decentralized training theoretically~\citep{lian2017can, tang2018d}, this choice is usually for the sake of proof simplicity (see footnote on Page 6 of \citet{lian2017can}). In practice it is tempting to choose diminishing stepsize in the learning rate scheduler, since the model training often benefits from large stepsizes, a phenomenon that has attracted a lot of attention recently in deep learning community~\citep{lewkowycz2020large, cohen2021gradient}. We thus compare D-SGD, D-SGT, D-ASAGT with D-NASA using diminishing stepsizes with a tunable hyperparameter.

For D-SGD and D-SGT, we test the algorithm with diminishing stepsizes $\eta_t=\eta\sqrt{n/t}$ as suggested by \citet{lian2017can,koloskova2021improved}; For D-ASAGT, we test the algorithm with stepsizes $\eta_t=\eta\sqrt{n/t}$ and $\alpha_t=\min\{\sqrt{n/t}, 0.3\}$ as suggested in their experiments~\cite{pmlr_v216_xiao23a}; For D-NASA we take $\eta_t=n^{1/4}/t^{3/4}$ and $\alpha_t=\sqrt{n/t}$ based on our theoretical analysis. We conduct a simple grid search for D-SGD, D-SGD and D-ASAGT to determine and use the best choices of $\eta$ for each algorithms. The convergence result is shown in Figure \ref{fig:synthetic_spike_model}. Among all algorithms, the test loss of D-NASA decreases with oscillations, presenting the catapults~\citep{lewkowycz2020large} and Edge of Stability (EOS)~\citep{cohen2021gradient} phenomena, two closely related large-stepsize regimes in which the training converges non-monotonically with oscillations and usually generalize better than small-stepsize settings~\citep{lewkowycz2020large, cohen2021gradient, arora2022understanding, ahn2022learning}. Furthermore, we observe that the test loss of our algorithm is much lower than other baselines, indicating superior generalization performance. We emphasize that different from the large-stepsize training setup in the literature, our Algorithm presents the catapults and EOS without any hyperparameter tuning.

\begin{figure*}[!h]
    \begin{center}
    \subfigure[]{\includegraphics[width=0.44\textwidth]{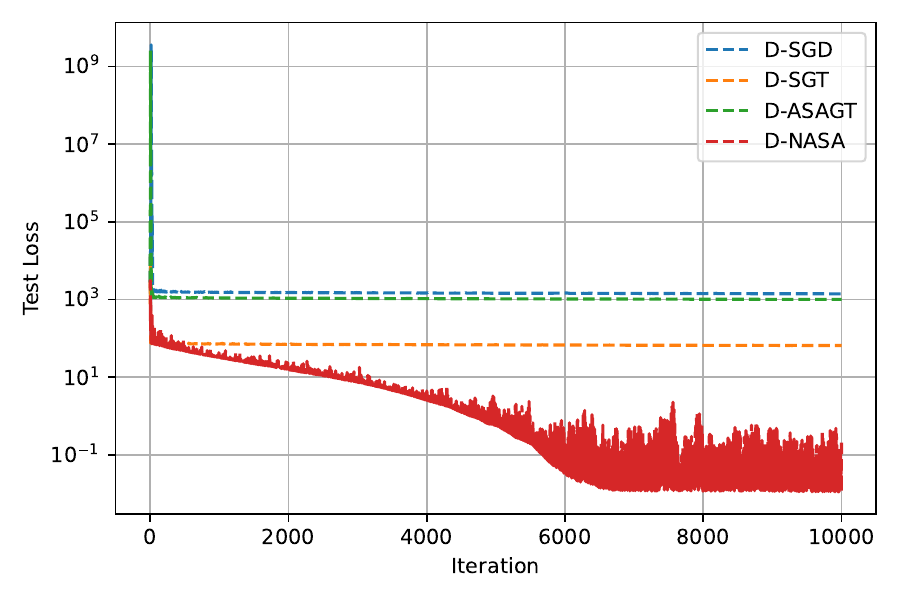}\label{fig:synthetic_spike_testloss}}
    \subfigure[]{\includegraphics[width=0.44\textwidth]{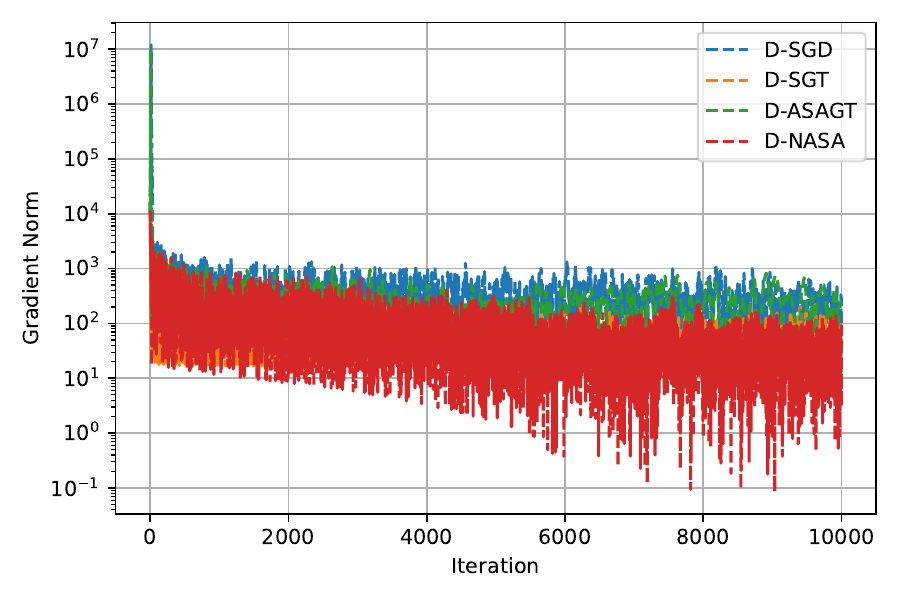}\label{fig:synthetic_spike_gradnorm}}
    
    \caption{Convergence curve for D-SGD, D-SGD, D-ASAGT and D-NASA for solving \eqref{eq_synthetic_ls} under the spike model.}
    \label{fig:synthetic_spike_model}
    \end{center}
\end{figure*}

\subsection{Real-world data experiments}\label{sec_real_data_experiments}
We utilize the code framework in \citet{mancino2023proximal} where we compare D-NASA with D-SGD, D-SGT, D-ASAGT for solving the classification problem:
\begin{equation}\label{eq_real_data_classification}
\min _{\theta \in \mathbb{R}^d} \frac{1}{n} \sum_{i=1}^n \frac{1}{\left|\mathcal{D}_i\right|} \sum_{(x, y) \in \mathcal{D}_i} \ell(f(x ; \theta), y)
\end{equation}
on MNIST, a9a and miniboone datasets\footnote{Available at \url{https://www.openml.org}}. Here $\ell$ denotes the cross-entropy loss, and $f$ represents a neural network parameterized by $\theta$ with $x$ being its input data. $\mathcal{D}_i$ is the training set only available to agent $i$. We use a $2$-layer perception model on a9a and miniboone, and the LeNet architecture~\cite{lecun2015lenet} for the MNIST dataset. We take $n=8$ which connect in the form of a random graph ($\rho=0.375$) for all three datasets\footnote{To make sure that the graph is connected, we set the probability of each two node being connected as $0.8$. We refer to Appendix \ref{appendix_experiment} for more graph designs due to page limits.}. The data is divided evenly to $n=8$ devices (CPUs) and using mpi4py interface to communicate the computation results. The batch-sizes are fixed to be 32.

Similar to the synthetic data, for D-SGD and D-SGT, again we test the algorithm with diminishing stepsizes $\eta_t=\eta\sqrt{n/t}$; For D-ASAGT, we test the algorithm with $\eta_t=\eta\sqrt{n/t}$ and $\alpha_t=\min\{\sqrt{n/t}, 0.3\}$; For D-NASA we again take $\eta_t=n^{1/4}/t^{3/4}$ and $\alpha_t=\sqrt{n/t}$ based on our theory. Figure \ref{fig:acc_learning_rate} shows the test accuracy under different stepsizes $\eta\in \{0.005, 0.01, 0.5, 1, 5, 10, 50, 100\}$, where the dashed horizontal line is the result for D-NASA. We can see that D-NASA yields comparable numerical results without tuning any parameters, and other three algorithms can also work well under certain parameter choices\footnote{It can be seen that D-SGD, D-SGT and D-ASAGT all seem to work well when $\eta$ is around 10. We believe the main reason is that all the datasets we tested are normalized and the Lipschitz smooth constant are fairly similar for all three datasets.}.

\begin{figure*}[!h]
    \begin{center}
    \subfigure[MNIST]{\includegraphics[width=0.32\textwidth]{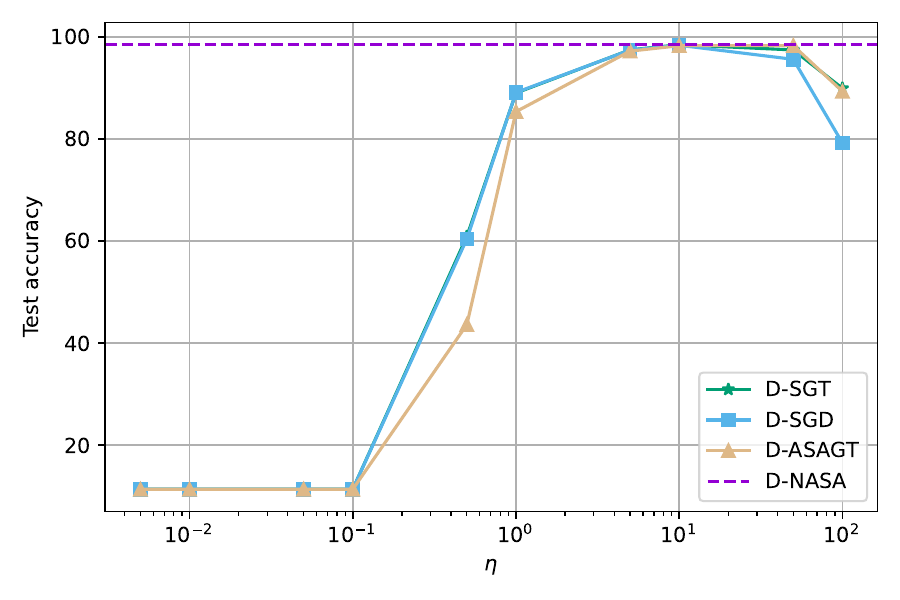}}
    \subfigure[a9a]{\includegraphics[width=0.32\textwidth]{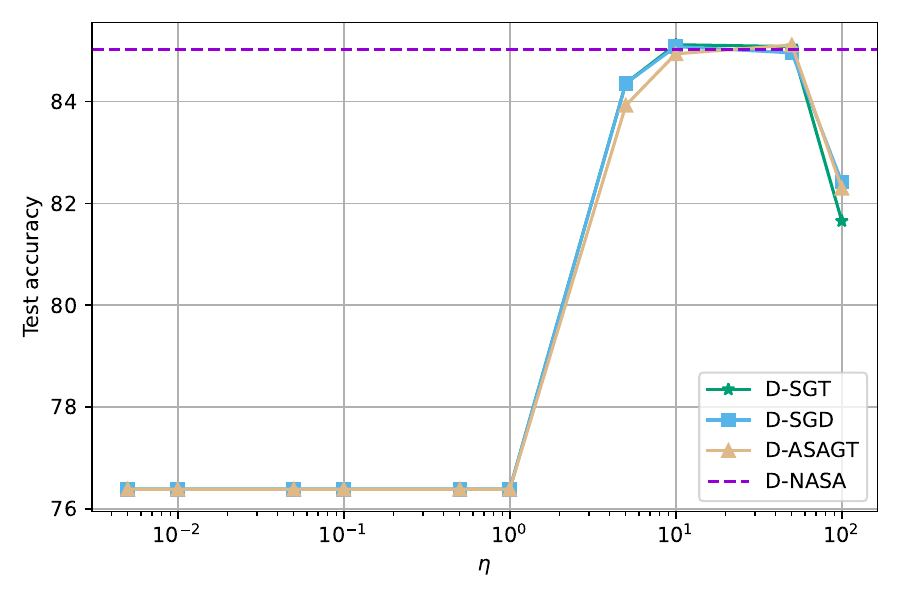}}
    \subfigure[miniboone]{\includegraphics[width=0.32\textwidth]{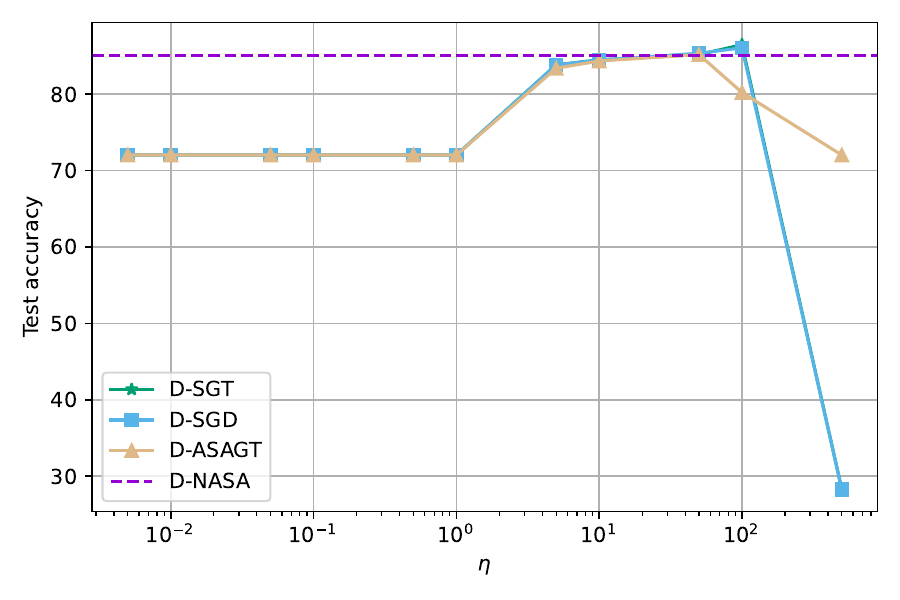}}
    
    \caption{The testing accuracy of the outputs from different algorithms with respect to different choices of learning rates.}
    \label{fig:acc_learning_rate}
    \end{center}
\end{figure*}

\section{Conclusion}
In this paper we propose D-NASA, a problem-parameter-free decentralized stochastic optimization algorithm and give its finite-time convergence analysis. Moreover, we showcase that in comparison to other baselines, our algorithm demonstrates superior generalization without tedious hyperparameter tuning process, thus having great potential for large scale machine learning problems. It would be interesting to explore parameter-free convergence in convex, also the nonsmooth regimes.

% \section*{Acknowledgements}

\clearpage

\bibliography{reference}
\bibliographystyle{icml2024}

%%%%%%%%%%%%%%%%%%%%%%%%%%%%%%%%%%%%%%%%%%%%%%%%%%%%%%%%%%%%%%%%%%%%%%%%%%%%%%%
%%%%%%%%%%%%%%%%%%%%%%%%%%%%%%%%%%%%%%%%%%%%%%%%%%%%%%%%%%%%%%%%%%%%%%%%%%%%%%%
% APPENDIX
%%%%%%%%%%%%%%%%%%%%%%%%%%%%%%%%%%%%%%%%%%%%%%%%%%%%%%%%%%%%%%%%%%%%%%%%%%%%%%%
%%%%%%%%%%%%%%%%%%%%%%%%%%%%%%%%%%%%%%%%%%%%%%%%%%%%%%%%%%%%%%%%%%%%%%%%%%%%%%%
\newpage
\appendix
\onecolumn

\begin{center}
    \noindent\rule{\textwidth}{4pt} \vspace{-0.2cm}
    
    \LARGE \textbf{Appendix}
    
    \noindent\rule{\textwidth}{1.2pt}
\end{center}

\section{Details of experiments}\label{appendix_experiment}
Our experiments are performed on Amazon AWS EC2 g5.4xlarge cluster which consists of 16 vCPUs with 64 GiB memory and NVIDIA A10G GPU with 24 GiB memory. All the experiments are conducted on CPU where 8 CPU are used to imitate 8 different nodes. The network topology is chosen as Figure \ref{fig:different_graphs}.

\begin{figure*}[h]
    \centering
    \begin{tikzpicture}[shorten >=1pt,-] % ring
      \tikzstyle{vertex}=[circle,fill=black!25,minimum size=6pt,inner sep=2pt]
      \node[vertex] (G_1) at (0, 1) {};
      \node[vertex] (G_2) at (0.707, 0.707)   {};
      \node[vertex] (G_3) at (1, 0)  {};
      \node[vertex] (G_4) at (0.707, -0.707)  {};
      \node[vertex] (G_5) at (0, -1)  {};
      \node[vertex] (G_6) at (-0.707, -0.707)  {};
      \node[vertex] (G_7) at (-1, 0)  {};
      \node[vertex] (G_8) at (-0.707, 0.707)  {};
      \draw (G_1) -- (G_2) -- (G_3) -- (G_4) -- (G_5) -- (G_6) -- (G_7) -- (G_8) -- (G_1) -- cycle;
      \node [below=0.3cm, align=flush center,text width=3cm] at (G_5)
        {
            $\text{Ring } \rho=0.805$
        };
    \end{tikzpicture}
    \begin{tikzpicture}[shorten >=1pt,-] % random
      \tikzstyle{vertex}=[circle,fill=black!25,minimum size=6pt,inner sep=2pt]
      \node[vertex] (v1) at (0, 1) {};
      \node[vertex] (v2) at (0.707, 0.707)   {};
      \node[vertex] (v3) at (1, 0)  {};
      \node[vertex] (v4) at (0.707, -0.707)  {};
      \node[vertex] (v5) at (0, -1)  {};
      \node[vertex] (v6) at (-0.707, -0.707)  {};
      \node[vertex] (v7) at (-1, 0)  {};
      \node[vertex] (v8) at (-0.707, 0.707)  {};
      \draw (v1) -- (v2) -- (v3) -- (v6) -- (v4) -- (v5) -- (v7) -- (v8) -- (v1) --cycle;
      \draw (v1) -- (v3) -- (v7) -- (v4) -- (v8) -- (v2) -- (v5) -- (v6) -- (v1) -- cycle;
      \draw (v1) -- (v4) -- (v2) -- (v6) -- (v7) -- (v1) --cycle;
      \draw (v1) -- (v5) -- (v8) -- cycle;
      \draw (v3) -- (v8) -- cycle;
      \node [below=0.3cm, align=flush center,text width=3cm] at (v5)
        {
            $\text{Random } \rho=0.375$
        };
    \end{tikzpicture}
    \begin{tikzpicture}[shorten >=1pt,-] % ladder
      \tikzstyle{vertex}=[circle,fill=black!25,minimum size=6pt,inner sep=2pt]
      \def\mypoints{%
          (-0.5, 1), (-0.5, 0.333), (-0.5, -0.333), (-0.5, -1),
          (0.5, 1), (0.5, 0.333), (0.5, -0.333), (0.5, -1)
      };
      \foreach \x [count=\xi] in \mypoints {
            \node[vertex] (\xi) at \x {};
      };
      \draw (1) -- (2) -- (3) -- (4);
      \draw (5) -- (6) -- (7) -- (8);
      \draw (1) -- (5);
      \draw (2) -- (6);
      \draw (3) -- (7);
      \draw (4) -- (8);
      \node [below=0.3cm, align=flush center,text width=3cm] at (v5)
        {
            $\text{Ladder } \rho=0.892$
        };
    \end{tikzpicture}
    \begin{tikzpicture}[shorten >=1pt,-] % complete
      \tikzstyle{vertex}=[circle,fill=black!25,minimum size=6pt,inner sep=2pt]
      \def\mypoints{%
          (0, 1), (0.707, 0.707), (1, 0), (0.707, -0.707),
          (0, -1), (-0.707, -0.707), (-1, 0), (-0.707, 0.707)
      };
      \foreach \x [count=\xi] in \mypoints {
            \foreach \y [count=\zeta] in \mypoints {
                \ifnum\zeta>\xi
                    \draw \x -- \y ;
                \fi
            };
      };
      \foreach \x [count=\xi] in \mypoints {
            \node[vertex] (\xi) at \x {};
      };
      \node [below=0.3cm, align=flush center,text width=3cm] at (v5)
        {
            $\text{Complete } \rho=0$
        };
    \end{tikzpicture}
    \caption{Network topology for $n=8$. The four graphs represent the ring, (an instance of) the random, the ladder and the complete graph.}
    \label{fig:different_graphs}
\end{figure*}
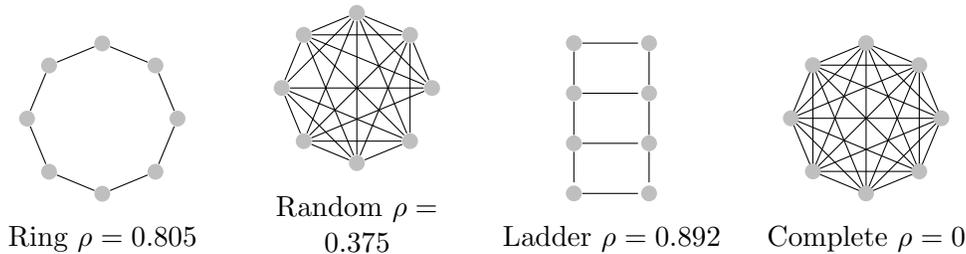

We now present the additional results for testing D-SGD, D-SGT, D-ASAGT, D-NASA on \eqref{eq_real_data_classification} with a9a data over different network topology as specified in Figure \ref{fig:different_graphs}. The hyperparameters follow exactly the same as in Section \ref{sec_real_data_experiments}. The results are presented in Figure \ref{fig:a9a_acc_learning_rate}. It can be seen that D-NASA achieves competitive testing accuracy under almost every network topology choice.

\begin{figure*}[ht!]
    \begin{center}
    \subfigure[Ring]{\includegraphics[width=0.44\textwidth]{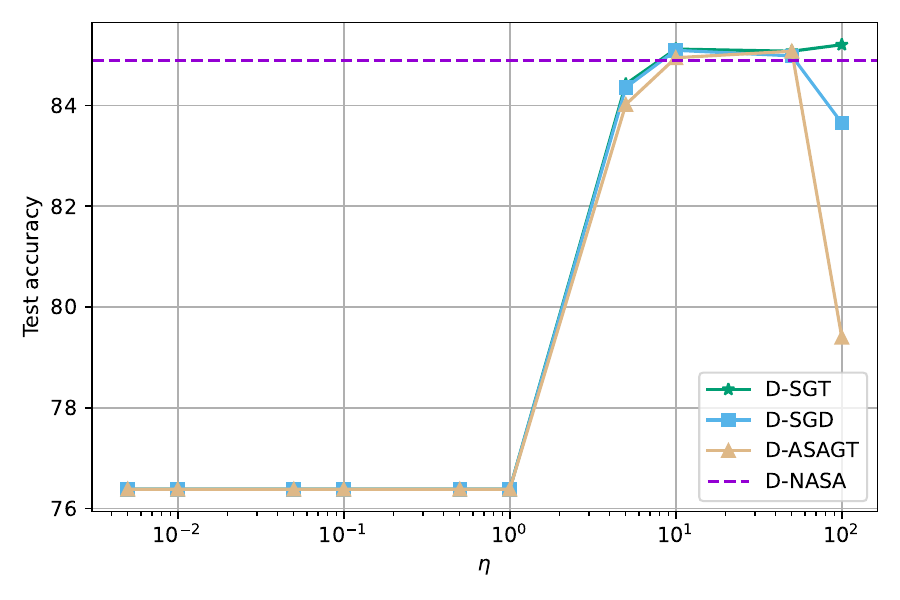}}
    \subfigure[Random]{\includegraphics[width=0.44\textwidth]{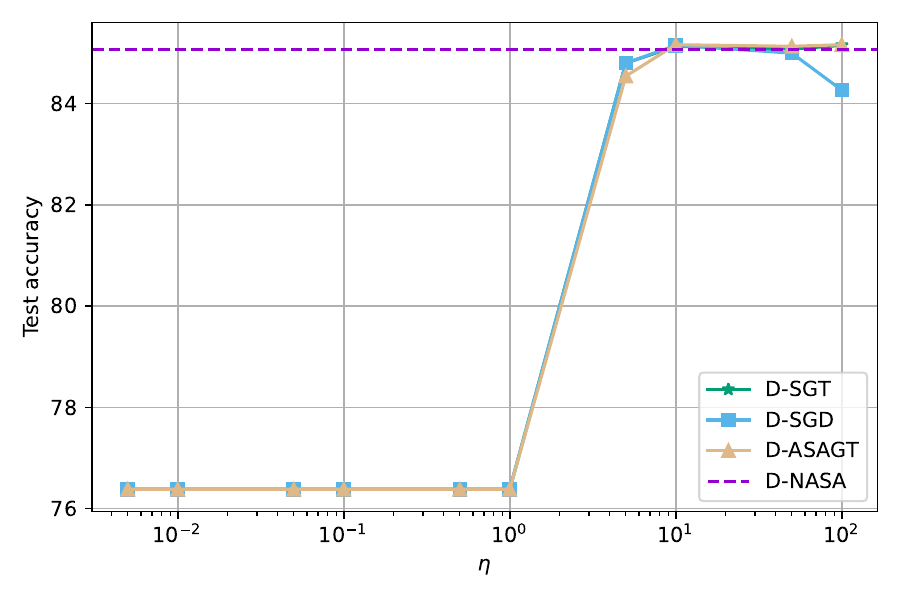}}
    \subfigure[Ladder]{\includegraphics[width=0.44\textwidth]{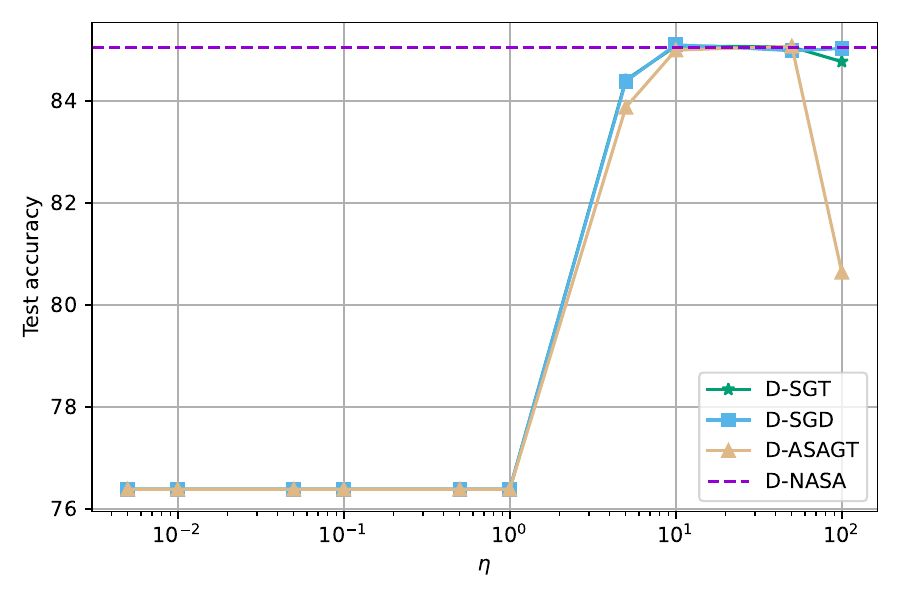}}
    \subfigure[Complete]{\includegraphics[width=0.44\textwidth]{a9a_acc_different_stepsize_complete.pdf}}
    
    \caption{The testing accuracy of the outputs from different algorithms with respect to different choices of learning rates for a9a dataset. The four figures corresponds to four different network graphs as in Figure \ref{fig:different_graphs}.}
    \label{fig:a9a_acc_learning_rate}
    \end{center}
\end{figure*}

We also include the figures of the loss, accuracy, and stationarity curves of all algorithms. Figure \ref{fig:mnist} and \ref{fig:a9a} shows the training/testing curve with respect to training epoch or CPU time when applying the four algorithms to \eqref{eq_real_data_classification} with MNIST and a9a dataset. We show each algorithm with the \textbf{best choice of stepsizes} in the light of Figure \ref{fig:acc_learning_rate}. One can see that D-NASA achieves competitive rate of convergence with exactly the same stepsize choice as our theory, without tuning any parameter.

\begin{figure*}[ht!]
    \begin{center}
    \subfigure{\includegraphics[width=0.243\textwidth]{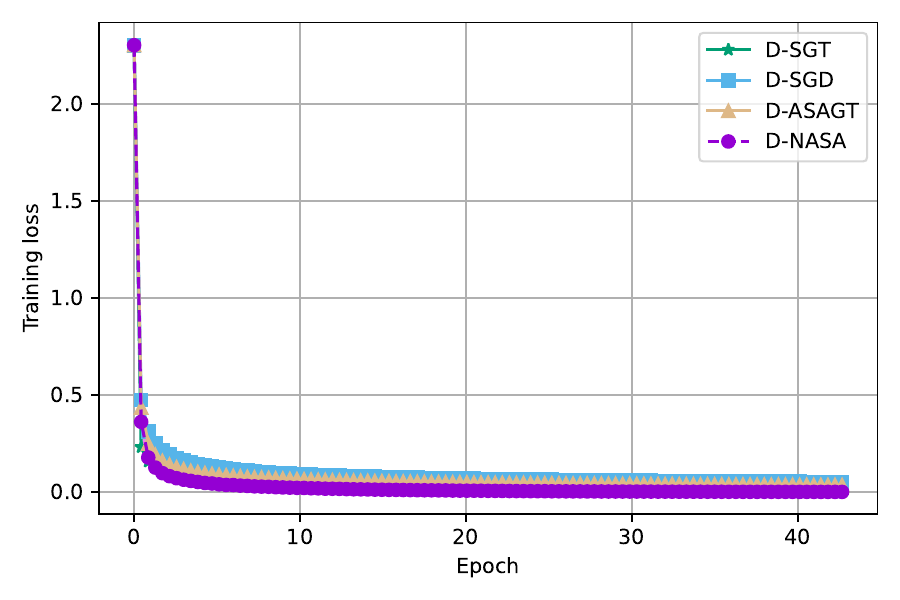}}
    \subfigure{\includegraphics[width=0.243\textwidth]{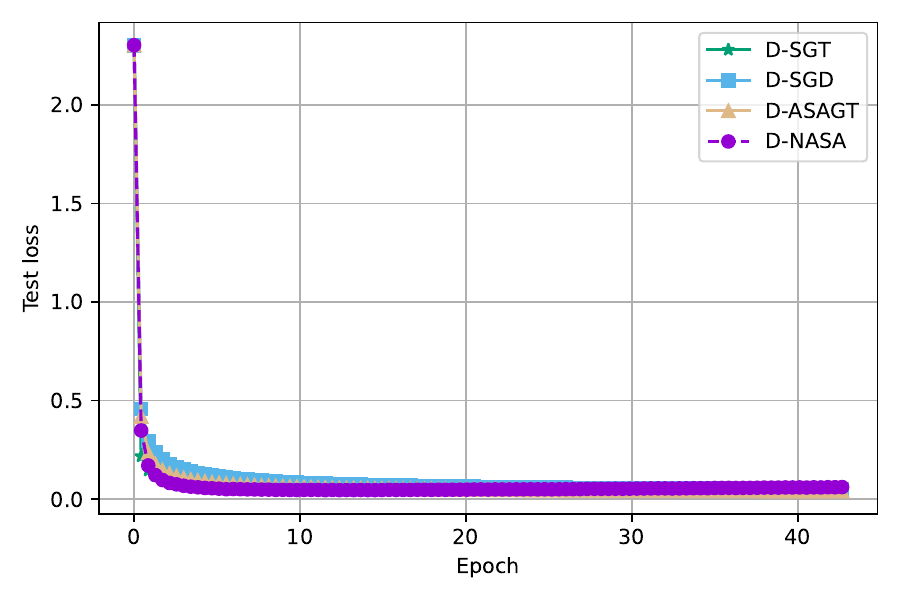}}
    \subfigure{\includegraphics[width=0.243\textwidth]{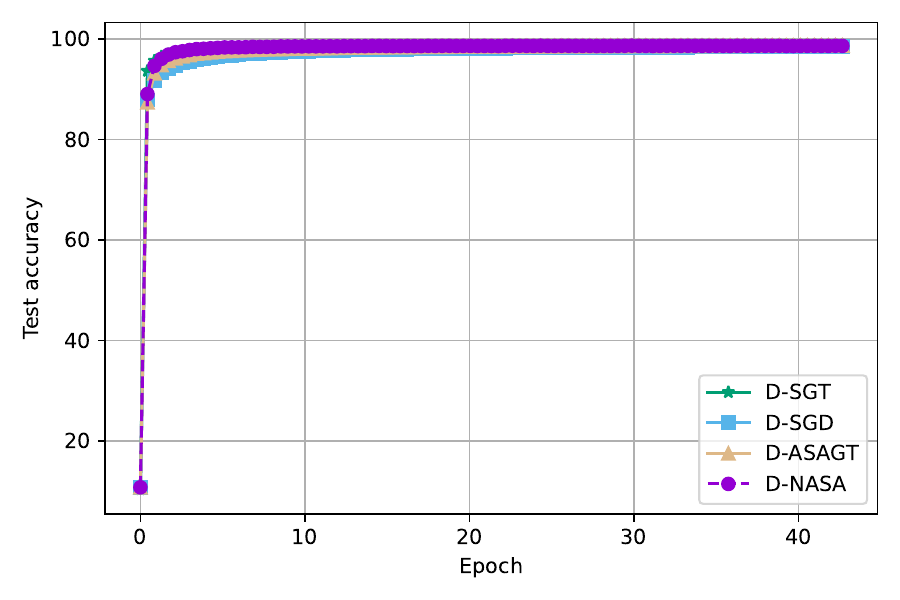}}
    \subfigure{\includegraphics[width=0.243\textwidth]{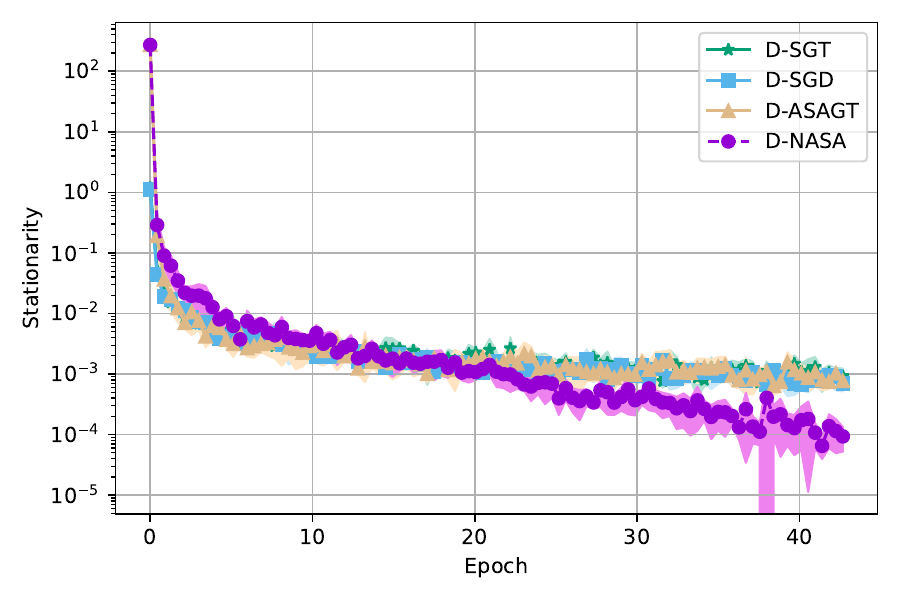}}

    % \setcounter{subfigure}{0}
    % \subfigure[]{\includegraphics[width=0.243\textwidth]{mnist_train_loss_time.pdf}}
    % \subfigure[]{\includegraphics[width=0.243\textwidth]{mnist_loss_time.pdf}}
    % \subfigure[]{\includegraphics[width=0.243\textwidth]{mnist_acc_time.pdf}}
    % \subfigure[]{\includegraphics[width=0.243\textwidth]{mnist_stat_time.pdf}}
    
    \caption{The convergence curve of D-SGD, D-SGT, D-ASAGT, D-NASA for the MNIST dataset, while the first three are at their \textbf{best stepsizes} (after the grid search as in Figure \ref{fig:acc_learning_rate}), and D-NASA follows the stepsize choice as in Remark \ref{rmk_dimishing_step}, i.e. $\eta_t=n^{1/4}/t^{3/4}$ and $\alpha_t=n^{1/2}/t^{1/2}$. The four columns are the curves for training loss, testing loss, testing accuracy and stationarity, respectively. The experiments are repeated and averaged for 10 times.} % The two rows are the plots corresponding to the Epoch (communication rounds, above) and CPU time (below).
    \label{fig:mnist}
    \end{center}
\end{figure*}

\begin{figure*}[ht!]
    \begin{center}
    \subfigure{\includegraphics[width=0.243\textwidth]{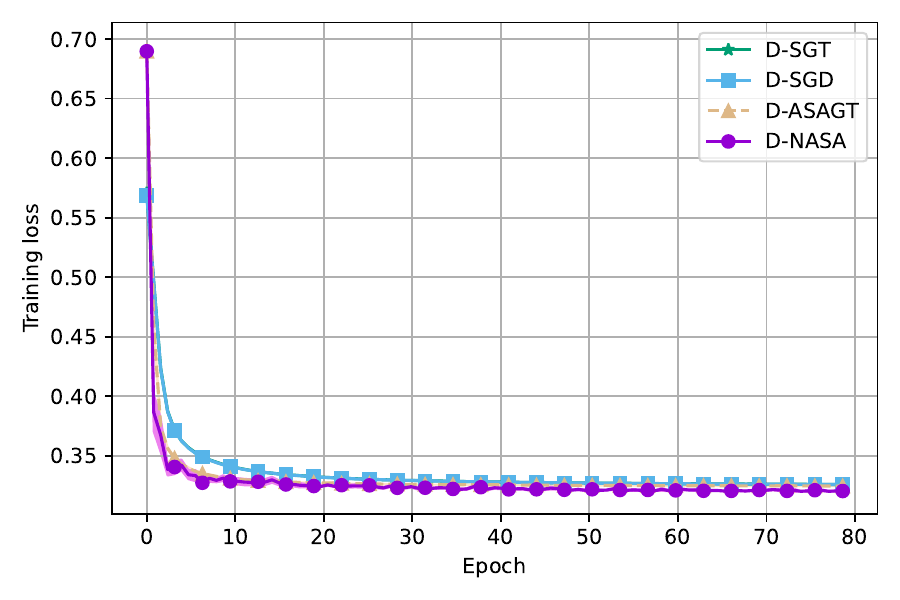}}
    \subfigure{\includegraphics[width=0.243\textwidth]{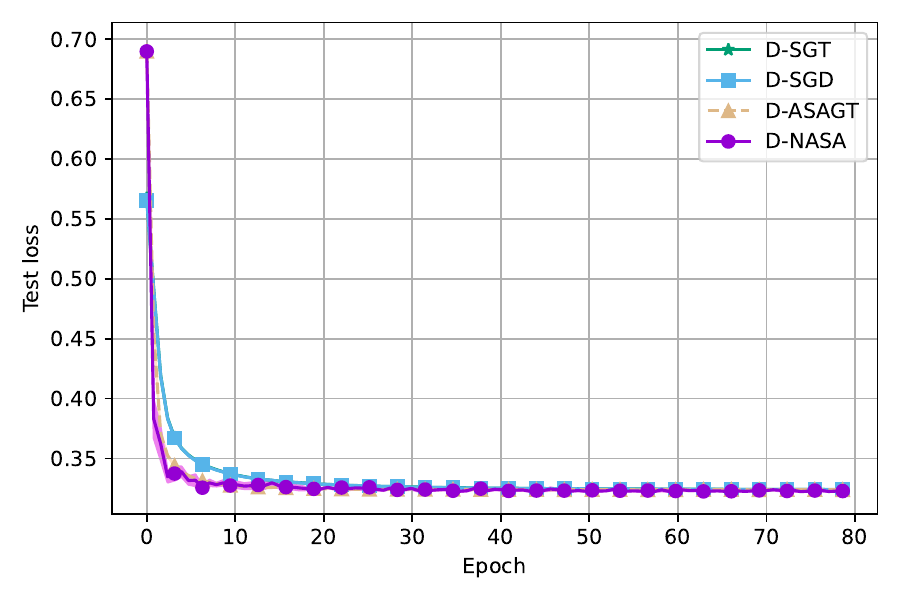}}
    \subfigure{\includegraphics[width=0.243\textwidth]{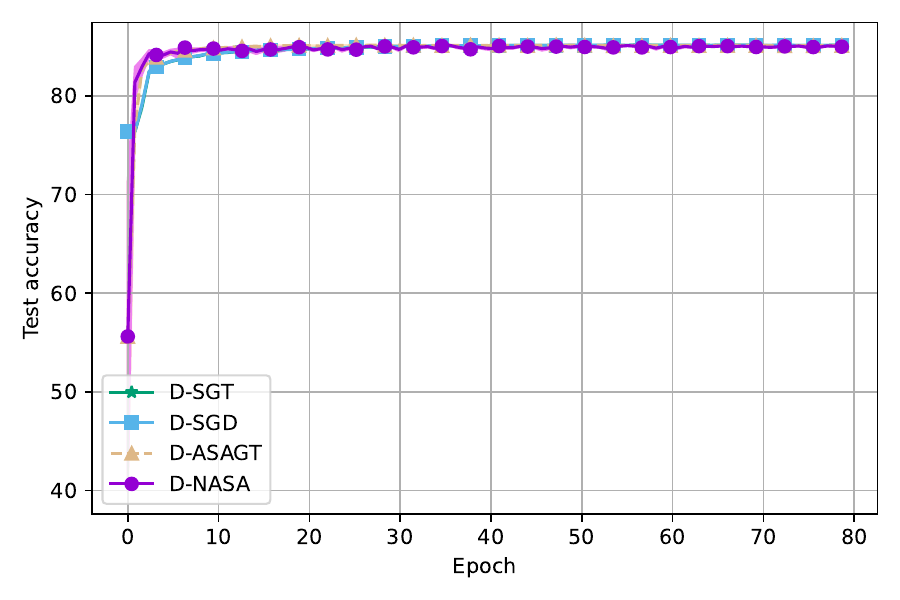}}
    \subfigure{\includegraphics[width=0.243\textwidth]{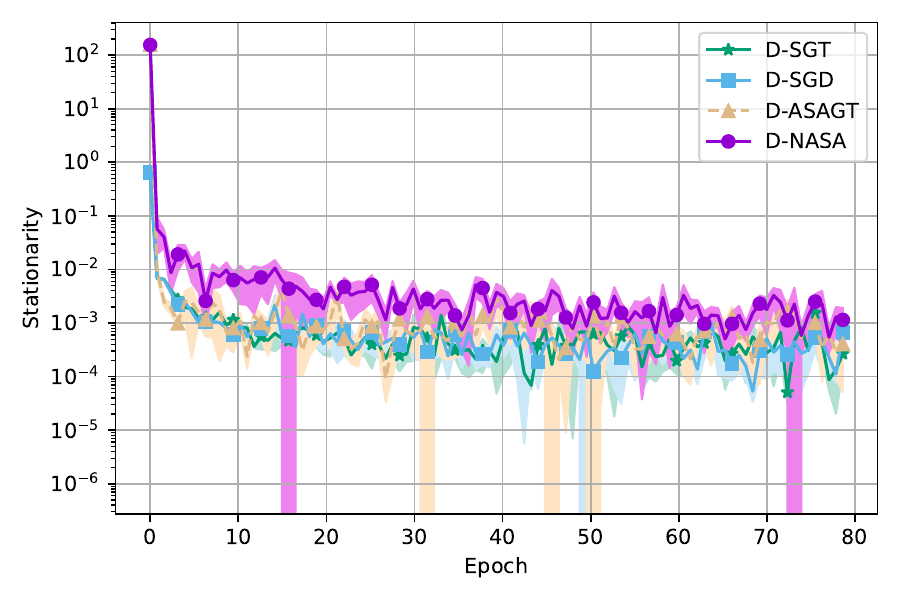}}

    % \setcounter{subfigure}{0}
    % \subfigure[]{\includegraphics[width=0.243\textwidth]{a9a_train_loss_time.pdf}}
    % \subfigure[]{\includegraphics[width=0.243\textwidth]{a9a_loss_time.pdf}}
    % \subfigure[]{\includegraphics[width=0.243\textwidth]{a9a_acc_time.pdf}}
    % \subfigure[]{\includegraphics[width=0.243\textwidth]{a9a_stat_time.pdf}}
    
    \caption{The convergence curve of D-SGD, D-SGT, D-ASAGT, D-NASA for the a9a dataset, while the first three are at their \textbf{best stepsizes} (after the grid search as in Figure \ref{fig:acc_learning_rate}), and D-NASA again follow the same choice as Figure \ref{fig:mnist}. The experiments are repeated and averaged for 10 times.}
    \label{fig:a9a}
    \end{center}
\end{figure*}

\section{Convergence analysis} %\label{sec_convergence}

\subsection{Parameter-free convergence theory for D-SGT}

From the update of Algorithm \ref{algo_decen_grad_tracking}, we have that:
\begin{equation}\label{eq_update_equations_dsgt}
\begin{aligned}
    &\X^{t+1}= \X^{t} - \eta_t \U^t,\ \Bar{x}^{t+1}=\Bar{x}^{t} - \eta_t \Bar{u}^t \\
    &\Bar{u}^{t}=\Bar{v}^{t}=\frac{1}{n}\sum_{i=1}^{n}\nabla F_i(x_i^{t},\xi_i^{t})
\end{aligned}
\end{equation}

The following descent lemma characterizes the difference between the function values of two consecutive iterates for Algorithm \ref{algo_decen_grad_tracking}:
\begin{lemma}\label{lemma_descent_dsgt}
    Suppose Assumption \ref{assump_l_smooth} and \ref{assump_l_continuous} holds. Algorithm \ref{algo_decen_grad_tracking} satisfies:
    \begin{equation*}
        \E[f(\bar{x}^{t+1})]- \E [f(\bar{x}^{t})] \leq -\frac{3\eta_t}{4} \E\norm{\nabla f(\Bar{x}^t)}^2 + 2\eta_t^2 L\frac{\sigma^2}{n} + (\eta_t+2\eta_t^2 L)\frac{L^2}{n}\E\|\X^t - \Bar{\X}^t\|^2 + \eta_t^2 L G^2
    \end{equation*}
    where the expectation is taken conditioned on $\mathcal{F}_{t-1}$.
\end{lemma}
\begin{proof}
    By the $L$-Lipschitz smooth of $f$ (Assumption \ref{assump_l_smooth}) we get:
    \begin{equation*}
    \begin{aligned}
        &f(\Bar{x}^{t+1}) - f(\Bar{x}^t) \\\leq& \nabla f(\Bar{x}^t)^\top(\Bar{x}^{t+1} - \Bar{x}^t) + \frac{L}{2}\|\Bar{x}^{t+1} - \Bar{x}^t\|^2 
        =-\eta_t\nabla f(\Bar{x}^t)^\top \Bar{u}^t + \frac{\eta_t^2 L}{2}\|\Bar{u}^t\|^2 \\
        = & -\eta_t \norm{\nabla f(\Bar{x}^t)}^2 -\eta_t\nabla f(\Bar{x}^t)^\top (\Bar{v}^t - \nabla f(\Bar{x}^t)) + \frac{\eta_t^2 L}{2}\|\Bar{v}^t\|^2 \\
        = & -\eta_t \norm{\nabla f(\Bar{x}^t)}^2 -\eta_t\nabla f(\Bar{x}^t)^\top (\Bar{v}^t - h^t + h^t - \nabla f(\Bar{x}^t)) + \frac{\eta_t^2 L}{2}\|\Bar{v}^t\|^2
    \end{aligned}
    \end{equation*}
    where $h^t:=\frac{1}{n}\sum_{i=1}^n\nabla f_i(x_i^t)$.

    Now taking the expectation conditioned on $\mathcal{F}_{t-1}$, we get
    \begin{equation*}
    \begin{aligned}
        &\E[f(\Bar{x}^{t+1}) - f(\Bar{x}^t)] \\
        \leq & -\eta_t \E\norm{\nabla f(\Bar{x}^t)}^2 -\eta_t\E[\nabla f(\Bar{x}^t)^\top (h^t - \nabla f(\Bar{x}^t))] + \frac{\eta_t^2 L}{2}\E\|\Bar{v}^t\|^2 \\
        \leq & -\eta_t \E\norm{\nabla f(\Bar{x}^t)}^2 +\eta_t/4\E\|\nabla f(\Bar{x}^t)\|^2 + \eta_t\E\|h^t - \nabla f(\Bar{x}^t)\|^2 + \eta_t^2 L\|\Bar{v}^t - \nabla f(\Bar{x}^t)\|^2 + \eta_t^2 L G^2 \\
        \leq & -\frac{3\eta_t}{4} \E\norm{\nabla f(\Bar{x}^t)}^2 + 2\eta_t^2 L\E\|\Bar{v}^t - h^t\|^2 + (\eta_t+2\eta_t^2 L)\E\|h^t - \nabla f(\Bar{x}^t)\|^2 + \eta_t^2 L G^2
    \end{aligned}
    \end{equation*}
    where the second inequality is by $\E[\Bar{v}^t]=h^t$, Cauchy-Schwarz, $a^\top b\leq 1/\gamma \|a\|^2+\gamma \|b^2\|$ and Assumption \ref{assump_l_continuous}, and the third is by $\|a+b\|^2\leq 2\|a\|^2+2\|b\|^2$. 

    Now taking the conditional expectation over $\mathcal{F}_{t-1}$ we get:
    $$
        \E\|\Bar{v}^t - h^t\|^2=\E\|\frac{1}{n}\sum_i(\nabla F_i(x_i^t,\xi_i^t) - \nabla f(x_i^t))\|^2=\frac{1}{n}\sum_i\E\|\nabla F_i(x_i^t,\xi_i^t) - \nabla f(x_i^t)\|^2\leq \frac{\sigma^2}{n}
    $$
    due to Assumption \ref{assump_bdd_variance}.

    As for the term $\|h^t - \nabla f(\Bar{x}^t)\|^2$, we have
    \begin{equation*}
        \|h^t - \nabla f(\Bar{x}^t)\|^2 = \|\frac{1}{n}(\sum_{i=1}^n\nabla f_i(x_i^t) - \nabla f_i(\Bar{x}^t))\|^2 \leq \frac{L^2}{n} \sum_i\|x_i^t - \Bar{x}^t\|^2 = \frac{L^2}{n}\|\X^t - \Bar{\X}^t\|^2.
    \end{equation*}
\end{proof}

We have the following lemma about the consensus error, that is, the average distance of each node to the global average.
\begin{lemma}\label{lemma_one_step_consensus_dsgt}
    For the update of Algorithm \ref{algo_decen_grad_tracking}, we have:
    \begin{equation*}
    \begin{aligned}
         \|\X^{t+1} - \Bar{\X}^{t+1}\|^2&\leq \frac{1+\rho}{2}\|\X^{t} - \Bar{\X}^{t}\|^2 + \eta_t^2 \frac{1+\rho^2}{1-\rho^2}\|\U^{t} - \Bar{\U}^t\|^2,\\
         \|\U^{t+1} - \Bar{\U}^{t+1}\|^2&\leq \frac{1+\rho}{2} \|\U^{t} - \Bar{\U}^{t}\|^2 + \frac{1+\rho^2}{1-\rho^2}  \|\V^{t+1} - \V^t\|^2.
    \end{aligned}
    \end{equation*}
\end{lemma}
\begin{proof}% [Proof of consensus error bound]
    Since 
    \begin{equation*}
    \begin{aligned}
        &\|\X^{t+1} - \Bar{\X}^{t+1}\|^2= \| (\X^{t}-\eta_t\U^t) W - (\Bar{x}^t-\eta_t\Bar{u}^t)\ones^\top \|^2 \\
        =& \| (\X^{t}-\eta_t\U^t) W - \frac{1}{n}(\X^{t}-\eta_t\U^t)\ones\ones^\top \|^2 = \| (\A^t- \A^t\frac{\ones^\top}{n})(W - \frac{\ones\ones^\top}{n})\|^2 \\
        \leq & \| \A^t- \A^t\frac{\ones^\top}{n}\|^2\|W - \frac{\ones\ones^\top}{n}\|_2^2 \\
        \leq &\rho^2 \| \A^t- \A^t\frac{\ones^\top}{n}\|^2 = \rho^2 \| (\X^{t} - \Bar{x}^t\ones^\top) -\eta_t (\U^t - \Bar{u}^t\ones^\top) \|^2 \\
        \leq & \rho^2(1+\frac{1}{c})\| \X^{t} - \Bar{x}^t\ones^\top\|^2 + \rho^2\eta_t^2(1+c)\|\U^t - \Bar{u}^t\ones^\top\|^2
    \end{aligned}
    \end{equation*}
    where $\A^t:=\X^{t}-\eta_t\U^t$. Taking $c=\frac{2\rho^2}{1-\rho^2}\geq 0$ gives the desired result. For the consensus error of $\U^t$ we could get it in a similar way.
\end{proof}

With the analysis of one step of the consensus error, we are readily to analyze the cumulative consensus error for the final convergence. To do this, we need the following technical lemma:
\begin{lemma}[Lemma 3.3 in \cite{pmlr_v216_xiao23a}]\label{lemma_technical_summation}
    Suppose we are given three sequences $\{a_n\}_{n=0}^{\infty}$, $\{c_n\}_{n=0}^{\infty}$, $\{\tau_n\}_{n=0}^{\infty}$, and a constant $r\in(0, 1)$ such that $a_k,b_k\geq 0$, $0=c_{-1}\leq c_{k+1}\leq c_{k}\leq 1$ and
    $$
        a_{k+1}\leq r a_k+b_k
    $$
    then we have
    \begin{equation*}
    \sum_{k=0}^K c_k a_k \leq \frac{1}{1-r}\left(c_0 a_0+\sum_{k=0}^K c_k b_k\right)
    \end{equation*}
    for any positive integer $K$.
\end{lemma}

Now we are ready to analyze the cumulative consensus error for Algorithm \ref{algo_decen_grad_tracking} as follows:
\begin{lemma}\label{lemma_concensus_summation_dsgt}
    For the update of Algorithm \ref{algo_decen_grad_tracking}, under Assumption \ref{assump_l_continuous} and \ref{assump_bdd_variance}, we have:
    \begin{equation*}
    \sum_{t=0}^{T-1}\frac{1}{n}\eta_t^\tau\E\|\X^{t} - \Bar{\X}^{t}\|^2\leq 10\Tilde{\rho} \sum_{t=0}^{T-1}\eta_t^{\tau+2} (\sigma^2+G^2)
    \end{equation*}
    where $\tau=0, 1$ or $2$ and
    $$
    \Tilde{\rho}:= \frac{\rho}{1-\rho}\frac{1+\rho^2}{1-\rho^2}.
    $$
    Note that $\Tilde{\rho}$ is greater than $0$.
\end{lemma}
\begin{proof}
    First by applying Lemma \ref{lemma_one_step_consensus_dsgt} and \ref{lemma_technical_summation} with $a_k=\frac{1}{n}\|\X^{k} - \Bar{\X}^{k}\|^2$, $b_k=\eta_k^2 \frac{1+\rho^2}{1-\rho^2}\frac{1}{n}\|\U^{k} - \Bar{\U}^k\|^2$ and $r=(1+\rho)/2$, we get
    \begin{equation}\label{eq_lemma_concensus_summation_dsgt_temp_1}
        \sum_{t=0}^{T-1}\frac{1}{n}\E\|\X^{t} - \Bar{\X}^{t}\|^2\leq \Tilde{\rho} \sum_{t=0}^{T-1}\eta_t^2\frac{1}{n}\E\|\U^{t} - \Bar{\U}^t\|^2
    \end{equation}

    Second, by applying Lemma \ref{lemma_one_step_consensus_dsgt} and \ref{lemma_technical_summation} again we get
    \begin{equation}\label{eq_lemma_concensus_summation_dsgt_temp_2}
        \sum_{t=0}^{T-1}\eta_t^2\frac{1}{n}\E\|\U^{t} - \Bar{\U}^{t}\|^2\leq \Tilde{\rho} \sum_{t=0}^{T-1}\eta_t^2\frac{1}{n}\E\|\V^{t+1} - \V^{t}\|^2
    \end{equation}

    Now we inspect the term $\V^{t+1}-\V^t$ following \cite{pmlr_v216_xiao23a}. We first have
    \begin{equation*}
    \begin{aligned}
    \V^{t+1}-\V^t= & \V^{t+1}-\E[\V^{t+1} \mid \mathscr{F}_t]-(\V^t-\E[\V^t \mid \mathscr{F}^{t-1}]) \\
    & +\E[\V^{t+1} \mid \mathscr{F}_t]-\nabla \mathbf{F}(\bar{x}^{t+1})+\nabla \mathbf{F}(\bar{x}^{t+1})-\nabla \mathbf{F}(\bar{x}^{t})+\nabla \mathbf{F}(\bar{x}^{t})-\E[\V^t \mid \mathscr{F}_{t-1}]
    \end{aligned}
    \end{equation*}
    where we use the notation $\nabla \mathbf{F}(x):=[\nabla f_1(x),...,\nabla f_n(x)]$ being the matrix of column gradient vectors. We thus have
    \begin{equation*}
    \begin{aligned}
    & \E\left\|\V^{t+1}-\V^t\right\|^2 \\
    \leq& 5\bigg\{\E\left\|\V^{t+1}-\E\left[\V^{t+1} \mid \mathscr{F}_t\right]\right\|^2+\E\left\|\V^t-\E\left[\V^t \mid \mathscr{F}_{t-1}\right]\right\|^2+\sum_{i=1}^n\E\left\|\nabla f_i(x_i^{t+1})-\nabla f_i(\bar{x}^{t+1})\right\|^2. \\
    &+\sum_{i=1}^n\E\left\|\nabla f_i(\bar{x}^{t+1})-\nabla f_i(\bar{x}^{t})\right\|^2+\sum_{i=1}^n\E\left\|\nabla f_i(x_i^{t})-\nabla f_i(\bar{x}^{t})\right\|^2\bigg\} \leq 10 n \sigma^2 + 60 n G^2
    \end{aligned}
    \end{equation*}
    where the first inequality uses Cauchy-Schwarz inequality, and the second utilizes Lipschitz continuity of each $f_i$. Plug this back to  \eqref{eq_lemma_concensus_summation_dsgt_temp_2} gives the result. For $k>0$ we can get the result in the exact same manner.
\end{proof}

Now we are ready to present our final convergence for Algorithm \ref{algo_decen_grad_tracking}, which we restate it here:
\begin{theorem}%\label{thm_fix_step_dsgt}
    Suppose Assumptions \ref{assump_l_smooth}, \ref{assump_l_continuous} and \ref{assump_bdd_variance} hold, also take $\eta_t=\eta T^{-1/2}$ for $\eta>0$, the update of Algorithm \ref{algo_decen_grad_tracking} satisfies:
    \begin{equation*}
        \frac{1}{T} \sum_{t=0}^{T-1}\E\|\nabla f(\bar{x}^t)\|^2 \leq \mathcal{O}\bigg( \frac{\Delta_0/\eta + (L\sigma^2/n+L G^2)\eta}{\sqrt{T}} + \frac{\Tilde{\rho} L^2 \eta^2}{T}(\sigma^2+G^2) \bigg).
    \end{equation*}

    If we take ($\eta_0=0$) $\eta_t=\eta t^{-1/2}$ for $\eta>0$, the update of Algorithm \ref{algo_decen_grad_tracking} satisfies:
    \begin{align*}
        \frac{1}{T} \sum_{t=0}^{T-1}\E\|\nabla f(\bar{x}^t)\|^2
        \leq \tilde{\mathcal{O}}\bigg( \frac{\Delta_0/\eta + (L\sigma^2/n+L G^2)\eta}{\sqrt{T}} + \frac{\Tilde{\rho} L^2 \eta^2}{T}(\sigma^2+G^2) \bigg).
    \end{align*}

    Note that we hide higher-order terms in $\mathcal{O}$ and $\log$ terms in $\tilde{\mathcal{O}}$.
\end{theorem}
\begin{proof}
    From Lemma \ref{lemma_descent_dsgt} we know that 
    \begin{equation*}
        \frac{3\eta_t}{4} \E\norm{\nabla f(\Bar{x}^t)}^2 \leq \E[f(\bar{x}^{t})]- \E [f(\bar{x}^{t+1})] + 2\eta_t^2 L\frac{\sigma^2}{n} + (\eta_t+2\eta_t^2 L)\frac{L^2}{n}\E\|\X^t - \Bar{\X}^t\|^2 + \eta_t^2 L G^2
    \end{equation*}
    sum up the above equation from $t=0$ to $T-1$ gives
    \begin{equation*}
    \begin{split}
        \sum_{t=0}^{T-1}\frac{3\eta_t}{4} \E\norm{\nabla f(\Bar{x}^t)}^2 \leq & \E[f(\bar{x}^{0})]- f^* + \frac{2 L\sigma^2}{n}\sum_{t=0}^{T-1}\eta_t^2 + \frac{L^2}{n}\sum_{t=0}^{T-1}(\eta_t+2\eta_t^2 L)\E\|\X^t - \Bar{\X}^t\|^2 + L G^2\sum_{t=0}^{T-1}\eta_t^2 \\
        \leq & \E[f(\bar{x}^{0})]- f^* + \frac{2 L\sigma^2}{n}\sum_{t=0}^{T-1}\eta_t^2 + 10\Tilde{\rho} L^2\sum_{t=0}^{T-1}(\eta_t^3+2\eta_t^4 L)(\sigma^2+G^2) + L G^2\sum_{t=0}^{T-1}\eta_t^2
    \end{split}
    \end{equation*}
    where we used Lemma \ref{lemma_concensus_summation_dsgt} for the second line.

    Now for the constant stepsize $\eta_t=\eta T^{-1/2}$, it's very straightforward to check that $\sum_t\eta_t^2=\eta^2$, $\sum_t\eta_t^3=\eta^3/\sqrt{T}$ and $\sum_t\eta_t^4=\eta^4/T$, therefore we get the following convergence result:
    \begin{equation*}
        \sum_{t=0}^{T-1}\frac{3\eta}{4\sqrt{T}} \E\norm{\nabla f(\Bar{x}^t)}^2
        \leq \Delta_0 + (\frac{2 L\sigma^2}{n}+L G^2)\eta^2 + 10\Tilde{\rho} L^2(\frac{\eta^3}{\sqrt{T}}+2 L\frac{\eta^4}{T})(\sigma^2+G^2)
    \end{equation*}
    i.e.
    \begin{equation*}
        \frac{\eta}{T}\sum_{t=0}^{T-1} \E\norm{\nabla f(\Bar{x}^t)}^2
        \leq \mathcal{O}\left(\frac{\Delta_0}{\sqrt{T}} + \frac{( L\sigma^2/n+L G^2)\eta^2}{\sqrt{T}} + \Tilde{\rho} L^2(\frac{\eta^3}{T}+ L\frac{\eta^4}{T^{3/2}})(\sigma^2+G^2)\right)
    \end{equation*}
    where $\Delta_0:=\E[f(\bar{x}^{0})]- f^*$. This gives the first result in the theorem.

    For the diminishing stepsize ($\eta_0=0$) $\eta_t=\eta t^{-1/2}$ for $\eta>0$, it's again very straightforward to check that $\sum_t\eta_t^2\leq\eta^2\log(T)$, $\sum_t\eta_t^3\leq\eta^3/\sqrt{T}$ and $\sum_t\eta_t^4\leq\eta^4/T$, therefore we get the following convergence result:
    \begin{equation*}
        \sum_{t=0}^{T-1}\frac{3\eta}{4\sqrt{T}} \E\norm{\nabla f(\Bar{x}^t)}^2
        \leq \Delta_0 + (\frac{2 L\sigma^2}{n}+L G^2)\eta^2\log(T) + 10\Tilde{\rho} L^2(\frac{\eta^3}{\sqrt{T}}+2 L\frac{\eta^4}{T})(\sigma^2+G^2)
    \end{equation*}
    which results in the second line of the result.
\end{proof}

% \begin{remark}
%     The rate of $\mathcal{O}(1/\sqrt{T})$ matches the lower bound for nonconvex stochastic optimization~\cite{arjevani2023lower}. It's also worth noticing that we are not able to choose the parameter $\eta$ to achieve a linear speedup effect (even if we assume access to $n$, the number of nodes), due to the existence of the term related to $G$. In fact, 
    
%     We also remind the reader that if we assume the access of Lipschitz smooth constant $L$, one can achieve linear speedup for D-SGT as in~\cite{zhang2019decentralized, xin2021improved, koloskova2021improved}. This urges designing a new algorithm that could achieve linear speedup under parameter-free setting for decentralized optimization.
% \end{remark}

\subsection{Parameter-free convergence theory for D-NASA}

From the update of Algorithm \ref{algo_decen_normalized_averaged_grad_tracking} we have that:
\begin{equation}\label{eq_update_equations}
\begin{aligned}
    &\X^{t+1}= \X^{t} - \eta_t \hat{\Z}^t,\ \Bar{x}^{t+1}=\Bar{x}^{t} - \frac{1}{n}\sum_{i=1}^{n}\frac{\eta_t}{\|z_i^t\|}z_i^t\\
    &\Z^{t+1}= (1-\alpha_t)\Z^t W + \alpha_t \U^{t+1} W,\ \Bar{z}^{t+1}=(1-\alpha_t)\Bar{z}^t+\alpha_t \Bar{u}^{t+1}\\
    &\Bar{u}^{t+1}=\Bar{v}^{t+1}=\frac{1}{n}\sum_{i=1}^{n}\nabla F_i(x_i^{t},\xi_i^{t})
\end{aligned}
\end{equation}
where
$$
\hat{\Z}^t:=\bigg[\frac{z_1^t}{\|z_1^t\|},...,\frac{z_n^t}{\|z_n^t\|}\bigg]
$$
is the collections of column vectors where each column is normalized $z_i^t$.

The following descent lemma characterizes the difference between the function value of two consecutive iterates for Algorithm \ref{algo_decen_normalized_averaged_grad_tracking}: % is a lemma adopted from Lemma 12 in \cite{hubler2023parameter}:
\begin{lemma}\label{lemma_descent}
    Suppose Assumption \ref{assump_l_smooth} holds. Algorithm \ref{algo_decen_normalized_averaged_grad_tracking} satisfies:
    \begin{equation}
        f(\bar{x}^{t+1}) - f(\bar{x}^{t}) \leq -\eta_t\|\nabla f(\bar{x}^t)\| + 2\eta_t \|\Bar{z}^t - \nabla f(\bar{x}^t)\| + \frac{\eta_t}{n}\sum_{i=1}^{n}\|z_{i}^{t} - \Bar{z}^t\| + \frac{\eta_t^2 L}{2}
    \end{equation}
\end{lemma}
\begin{proof}
    By the $L$-Lipschitz smooth of $f$ (Assumption \ref{assump_l_smooth}) we get:
    \begin{equation*}
    \begin{aligned}
        f(\Bar{x}^{t+1}) - f(\Bar{x}^t) \leq& \nabla f(\Bar{x}^t)^\top(\Bar{x}^{t+1} - \Bar{x}^t) + \frac{L}{2}\|\Bar{x}^{t+1} - \Bar{x}^t\|^2 \\
        =& -\eta_t \nabla f(\Bar{x}^t)^\top\bigg(\frac{1}{n}\sum_{i=1}^{n}\frac{z_i^t}{\|z_i^t\|}\bigg) + \frac{\eta_t^2 L}{2}\bigg\|\frac{1}{n}\sum_{i=1}^{n}\frac{z_i^t}{\|z_i^t\|}\bigg\|^2 \\
        \leq & -\eta_t \nabla f(\Bar{x}^t)^\top\bigg(\frac{1}{n}\sum_{i=1}^{n}\frac{z_i^t}{\|z_i^t\|}\bigg) + \frac{\eta_t^2 L}{2} \\
        = & -\eta_t (\nabla f(\Bar{x}^t) - \Bar{z}^t)^\top\bigg(\frac{1}{n}\sum_{i=1}^{n}\frac{z_i^t}{\|z_i^t\|}\bigg) - \eta_t(\Bar{z}^t)^\top\bigg(\frac{1}{n}\sum_{i=1}^{n}\frac{z_i^t}{\|z_i^t\|}\bigg) + \frac{\eta_t^2 L}{2} \\
        = & -\eta_t (\nabla f(\Bar{x}^t) - \Bar{z}^t)^\top\bigg(\frac{1}{n}\sum_{i=1}^{n}\frac{z_i^t}{\|z_i^t\|}\bigg) - \eta_t(\Bar{z}^t)^\top\bigg(\frac{1}{n}\sum_{i=1}^{n}\frac{z_i^t}{\|z_i^t\|} -\frac{\Bar{z}^t}{\|\Bar{z}^t\|}\bigg) - \eta_t\|\Bar{z}^t\| + \frac{\eta_t^2 L}{2} \\
        \leq & 2\eta_t \|\nabla f(\Bar{x}^t) - \Bar{z}^t\| - \eta_t \|\nabla f(\Bar{x}^t)\| + \eta_t \|\Bar{z}^t\|\bigg\|\frac{1}{n}\sum_{i=1}^{n}\frac{z_i^t}{\|z_i^t\|} -\frac{\Bar{z}^t}{\|\Bar{z}^t\|}\bigg\| + \frac{\eta_t^2 L}{2}
    \end{aligned}
    \end{equation*}
    where the second and third inequalities are by Cauchy-Schwarz inequality. It remains to bound the second last term in the last line. We have
    \begin{equation*}
    \begin{aligned}
        &\|\Bar{z}^t\|\bigg\|\frac{1}{n}\sum_{i=1}^{n}\frac{z_i^t}{\|z_i^t\|} -\frac{\Bar{z}^t}{\|\Bar{z}^t\|}\bigg\| = \frac{\|\Bar{z}^t\|}{n} \bigg\|\sum_{i=1}^{n}\frac{\|\Bar{z}^t\| - \|z_i^t\|}{\|\Bar{z}^t\|\|z_i^t\|}z_i^t\bigg\| \\
        \leq & \frac{\|\Bar{z}^t\|}{n}\sum_{i=1}^{n}\frac{|\|\Bar{z}^t\| - \|z_i^t\||}{\|\Bar{z}^t\|\|z_i^t\|}\|z_i^t\| = \frac{1}{n}\sum_{i=1}^{n} |\|\Bar{z}^t\| - \|z_i^t\||\leq \frac{1}{n}\sum_{i=1}^{n}\|z_i^t - \Bar{z}^t\|
    \end{aligned}
    \end{equation*}
    which concludes the proof.
\end{proof}

We have the following dual convergence. % , adopted from Lemma 3.8 in \cite{pmlr_v216_xiao23a}.
\begin{lemma}\label{lemma_consensus}
    We have
    \begin{equation}
        \Bar{z}^{t+1} - \nabla f(\Bar{x}^{t+1}) = (1-\alpha_t) (\Bar{z}^{t} - \nabla f(\Bar{x}^{t})) + \alpha_t(\delta_1^t + \delta_2^t + \delta_3^t)
    \end{equation}
    where
    \begin{equation*}
    \begin{aligned}
        \delta_1^t &= \frac{\nabla f(\Bar{x}^t) - \nabla f(\Bar{x}^{t+1})}{\alpha_t}, \\
        \delta_2^t &= \frac{1}{n}\sum_{i=1}^{n}\nabla f_i(x_i^t) - \nabla f(\Bar{x}^t), \\
        \delta_3^t &= \frac{1}{n}\sum_{i=1}^{n}\left( v_i^t - \nabla f_i(x_i^t) \right).
    \end{aligned}
    \end{equation*}
    
    Consequently, we get:
    \begin{equation}
        \E\|\Bar{z}^{t} - \nabla f(\Bar{x}^{t})\| \leq L\sum_{\tau=1}^{t}\beta_{(\tau+1):t}\eta_\tau + L\sum_{\tau=0}^{t}\beta_{(\tau+1):t}\alpha_{\tau}\sqrt{\frac{1}{n}\sum_{i=1}^{n}\E\|x_i^\tau - \Bar{x}^\tau\|^2} + \sigma\sqrt{\frac{1}{n}\sum_{\tau=0}^{t}\beta_{(\tau+1):t}^2\alpha_{\tau}^2}
    \end{equation}
    where we have the following conventions:
    $$\beta_t:=1-\alpha_t \ \mbox{ and } \ \beta_{a:b}:=\prod_{i=a}^{b}\beta_i$$
\end{lemma}
\begin{proof}
    By the update we know that $\Bar{u}^t=\Bar{v}^{t}=\frac{1}{n}\sum_{i=1}^{n} v_i^t$, thus
    \begin{equation*}
    \begin{aligned}
        \Bar{z}^{t+1} - \nabla f(\Bar{x}^{t+1}) &= (1-\alpha_t)\Bar{z}^{t} + \alpha_t\Bar{u}^t - \nabla f(\Bar{x}^{t+1})\\
        &= (1-\alpha_t) (\Bar{z}^{t} - \nabla f(\Bar{x}^{t})) + \alpha_t(\delta_1^t + \delta_2^t + \delta_3^t)
    \end{aligned}
    \end{equation*}

    Now repeat the above recursive relation we get
    \begin{equation*}
    \begin{aligned}
        \Bar{z}^{t} - \nabla f(\Bar{x}^{t}) &= (1-\alpha_{t-1}) (\Bar{z}^{t-1} - \nabla f(\Bar{x}^{t-1})) + + \alpha_{t-1}(\delta_1^{t-1} + \delta_2^{t-1} + \delta_3^{t-1}) \\
        &= \cdots \\ 
        &= \beta_{1:t}(\Bar{z}^{0} - \nabla f(\Bar{x}^{0})) + \sum_{\tau=1}^{t}\beta_{(\tau+1):t}\alpha_{\tau}(\delta_1^{\tau}+\delta_2^{\tau}+\delta_3^{\tau}) \\
        &= \sum_{\tau=1}^{t}\beta_{(\tau+1):t}\alpha_{\tau}\delta_1^{\tau} + \sum_{\tau=1}^{t}\beta_{(\tau+1):t}\alpha_{\tau}\delta_2^{\tau} + \sum_{\tau=0}^{t}\beta_{(\tau+1):t}\alpha_{\tau}\delta_3^{\tau} . 
    \end{aligned}
    \end{equation*}

    Therefore we get
    \begin{equation*}
    \begin{aligned}
        \E\|\Bar{z}^{t} - \nabla f(\Bar{x}^{t})\| &\leq \sum_{\tau=1}^{t}\beta_{(\tau+1):t}\alpha_{\tau}\E\|\delta_1^{\tau}\| + \sum_{\tau=1}^{t}\beta_{(\tau+1):t}\alpha_{\tau}\E\|\delta_2^{\tau}\| + \E\|\sum_{\tau=0}^{t}\beta_{(\tau+1):t}\alpha_{\tau}\delta_3^{\tau}\| \\
        &\leq L\sum_{\tau=1}^{t}\beta_{(\tau+1):t}\eta_\tau + L\sum_{\tau=0}^{t}\beta_{(\tau+1):t}\alpha_{\tau}\frac{1}{n}\sum_{i=1}^{n}\E\|x_i^\tau - \Bar{x}^\tau\| + \sqrt{\sum_{\tau=0}^{t}\beta_{(\tau+1):t}^2\alpha_{\tau}^2\E\|\delta_3^{\tau}\|^2}\\
        &\leq L\sum_{\tau=1}^{t}\beta_{(\tau+1):t}\eta_\tau + L\sum_{\tau=0}^{t}\beta_{(\tau+1):t}\alpha_{\tau}\sqrt{\frac{1}{n}\sum_{i=1}^{n}\E\|x_i^\tau - \Bar{x}^\tau\|^2} + \sigma\sqrt{\frac{1}{n}\sum_{\tau=0}^{t}\beta_{(\tau+1):t}^2\alpha_{\tau}^2}
    \end{aligned}
    \end{equation*}
    where the second term is due to smoothness of each $f_i$, also the last term is by $\E[\langle\delta_3^{\tau_1},\delta_3^{\tau_2}\rangle]=0$ for any $\tau_1\not=\tau_2$ (due to unbiased assumption) and
    \begin{equation*}
    \begin{aligned}
\sqrt{\sum_{\tau=0}^{t}\beta_{(\tau+1):t}^2\alpha_{\tau}^2\E\|\delta_3^{\tau}\|^2} &= \sqrt{\sum_{\tau=0}^{t}\beta_{(\tau+1):t}^2\alpha_{\tau}^2\frac{1}{n^2}\sum_{i=1}^{n}\E\left\| \nabla F_i(x_i^t,\xi_{i}^{t}) - \nabla f_i(x_i^t) \right\|^2}\\
&\leq \sqrt{\sum_{\tau=0}^{t}\beta_{(\tau+1):t}^2\alpha_{\tau}^2\frac{\sigma^2}{n}}
    \end{aligned}
    \end{equation*}
    since all the cross inner-product terms vanish due to unbiased assumption.
\end{proof}

We have the following lemma about the consensus error, that is, the average distance of each node to the global average.
\begin{lemma}\label{lemma_one_step_consensus}
    For the update of Algorithm \ref{algo_decen_normalized_averaged_grad_tracking}, we have:
    \begin{equation*}
    \begin{aligned}
         \|\X^{t+1} - \Bar{\X}^{t+1}\|^2&\leq \frac{1+\rho}{2}\|\X^{t} - \Bar{\X}^{t}\|^2 + \eta_t^2 \frac{1+\rho^2}{1-\rho^2}\|\hat{\Z}^{t} - \Bar{\hat{\Z}}^t\|^2,\\
         \|\hat{\Z}^{t} - \Bar{\hat{\Z}}^t\|^2&\leq n, \\
         \|\Z^{t+1} - \Bar{\Z}^{t+1}\|^2 &\leq \frac{1+\rho}{2} \|\Z^{t} - \Bar{\Z}^{t}\|^2 + \alpha_t^2\frac{1+\rho^2}{1-\rho^2}  \|\U^t - \Bar{\U}^t\|^2,\\
         \|\U^{t+1} - \Bar{\U}^{t+1}\|^2&\leq \frac{1+\rho}{2} \|\U^{t} - \Bar{\U}^{t}\|^2 + \frac{1+\rho^2}{1-\rho^2}  \|\V^{t+1} - \V^t\|^2 \\
         \E\|\V^{t+1} - \V^t\|^2&\leq 10 n \sigma^2+5n L^2 +5L^2 \E\left[\left\|\X^{t+1}-\bar{\X}^{t+1}\right\|^2+\left\|\X^{t}-\bar{\X}^{t}\right\|^2\right]
    \end{aligned}
    \end{equation*}
\end{lemma}
\begin{proof}% [Proof of consensus error bound]
    Since 
    \begin{equation*}
    \begin{aligned}
        &\|\X^{t+1} - \Bar{\X}^{t+1}\|^2= \| (\X^{t}-\eta_t\hat{\Z}^t) W - (\Bar{x}^t-\eta_t\Bar{\hat{z}}^t)\ones^\top \|^2 \\
        =& \| (\X^{t}-\eta_t\hat{\Z}^t) W - \frac{1}{n}(\X^{t}-\eta_t\hat{\Z}^t)\ones\ones^\top \|^2 = \| (\A^t- \A^t\frac{\ones^\top}{n})(W - \frac{\ones\ones^\top}{n})\|^2 \\
        \leq & \| \A^t- \A^t\frac{\ones^\top}{n}\|^2\|W - \frac{\ones\ones^\top}{n}\|_2^2 \\
        \leq &\rho^2 \| \A^t- \A^t\frac{\ones^\top}{n}\|^2 = \rho^2 \| (\X^{t} - \Bar{x}^t\ones^\top) -\eta_t (\hat{\Z}^t - \Bar{\hat{z}}^t\ones^\top) \|^2 \\
        \leq & \rho^2(1+\frac{1}{c})\| \X^{t} - \Bar{x}^t\ones^\top\|^2 + \rho^2\eta_t^2(1+c)\|\hat{\Z}^t - \Bar{\hat{z}}^t\ones^\top\|^2
    \end{aligned}
    \end{equation*}
    where $\A^t:=\X^{t}-\eta_t\hat{\Z}^t$. Taking $c=\frac{2\rho^2}{1-\rho^2}\geq 0$ gives the desired result. For the consensus error of $\Z^t$ and $\U^t$ we get it in similar ways. It remains to bound the consensus error of $\hat{\Z}^t$ and the term $\V^{t+1}-\V^t$. For the consensus error of $\hat{\Z}^t$, we have
    \begin{equation*}
        \|\hat{\Z}^{t} - \Bar{\hat{\Z}}^t\|^2 =\sum_{i=1}^{n}\| \frac{z_i^t}{\|z_i^t\|} - \frac{1}{n}\sum_{i=1}^{n}\frac{z_i^t}{\|z_i^t\|} \|^2
        \leq \sum_{i=1}^{n}\| \frac{z_i^t}{\|z_i^t\|}\|^2 \leq n
    \end{equation*}
    where we use
    $$
    \frac{1}{n}\sum_{i=1}^{n}\| v^i - \frac{1}{n}\sum_{i=1}^{n}v^i \|^2 =\frac{1}{n}\sum_{i=1}^{n}\| v^i\|^2 - \|\frac{1}{n}\sum_{i=1}^{n} v^i\|^2\leq \frac{1}{n}\sum_{i=1}^{n}\| v^i\|^2
    $$
    for any sequence of vectors $v^1,...,v^n$.

    Now we inspect the term $\V^{t+1}-\V^t$ similar to the proof of Lemma \ref{lemma_concensus_summation_dsgt}. We again have
    \begin{equation*}
    \begin{aligned}
    \V^{t+1}-\V^t= & \V^{t+1}-\E[\V^{t+1} \mid \mathscr{F}_t]-(\V^t-\E[\V^t \mid \mathscr{F}^{t-1}]) \\
    & +\E[\V^{t+1} \mid \mathscr{F}_t]-\nabla \mathbf{F}(\bar{x}^{t+1})+\nabla \mathbf{F}(\bar{x}^{t+1})-\nabla \mathbf{F}(\bar{x}^{t})+\nabla \mathbf{F}(\bar{x}^{t})-\E[\V^t \mid \mathscr{F}_{t-1}]
    \end{aligned}
    \end{equation*}
    where we use the notation $\nabla \mathbf{F}(x):=[\nabla f_1(x),...,\nabla f_n(x)]$ being the matrix of column gradient vectors. We thus have
    \begin{equation*}
    \begin{aligned}
    & \E\left\|\V^{t+1}-\V^t\right\|^2 \\
    \leq& 5\left\{\E\left\|\V^{t+1}-\E\left[\V^{t+1} \mid \mathscr{F}_t\right]\right\|^2+\E\left\|\V^t-\E\left[\V^t \mid \mathscr{F}_{t-1}\right]\right\|^2+\sum_{i=1}^n\E\left\|\nabla f_i(x_i^{t+1})-\nabla f_i(\bar{x}^{t+1})\right\|^2\right. \\
    &\left.+\sum_{i=1}^n\E\left\|\nabla f_i(\bar{x}^{t+1})-\nabla f_i(\bar{x}^{t})\right\|^2+\sum_{i=1}^n\E\left\|\nabla f_i(x_i^{t})-\nabla f_i(\bar{x}^{t})\right\|^2\right\} \\
    \leq& 5\left(2 n \sigma^2+n L^2 +L^2 \E\left[\left\|\X^{t+1}-\bar{\X}^{t+1}\right\|^2+\left\|\X^{t}-\bar{\X}^{t}\right\|^2\right]\right)
    \end{aligned}
    \end{equation*}
    where the first inequality uses Cauchy-Schwarz inequality, and the second utilizes Lipschitz smoothness of each $f_i$ also note that $\|\nabla f_i(\bar{x}^{t+1})-\nabla f_i(\bar{x}^{t})\|\leq L\|\Bar{\hat{z}}^t\|\leq L$.
\end{proof}

Now we are ready to analyze the cumulative consensus error for Algorithm \ref{algo_decen_normalized_averaged_grad_tracking} as follows:
\begin{lemma}\label{lemma_concensus_summation}
    For the update of Algorithm \ref{algo_decen_normalized_averaged_grad_tracking}, if decreasing sequences such that $0\leq\alpha_{t+1}\leq\alpha_{t}\leq 1$ and $0\leq\eta_{t+1}\leq\eta_{t}\leq 1$, we have:
    \begin{equation*}
    \begin{aligned}  
    \sum_{t=0}^{T-1}\frac{1}{n}\E\|\X^{t} - \Bar{\X}^{t}\|^2&\leq \Tilde{\rho}\sum_{t=0}^{T-1}\eta_t^2, \\
    \sum_{\tau=0}^{t}\alpha_{\tau}\sqrt{\frac{1}{n}\E\|\X^{\tau} - \Bar{\X}^{\tau}\|^2}  &\leq \Tilde{\rho}\sum_{\tau=0}^{t}\alpha_{\tau}\eta_{\tau}, \\
    \sum_{t=0}^{T-1}\frac{1}{n}\E\|\Z^{t} - \Bar{\Z}^{t}\|^2&\leq \Tilde{\rho}^2(10 \sigma^2+ 5L^2) \sum_{t=0}^{T-1}\alpha_t^2 + 2\Tilde{\rho}^3 \sum_{t=0}^{T} \alpha_t^2 \eta_t^2, \\
    \sum_{\tau=0}^{t}\eta_{\tau}\sqrt{\frac{1}{n}\E\|\Z^{\tau} - \Bar{\Z}^{\tau}\|^2}  &\leq \Tilde{\rho}^2\sqrt{10\sigma^2+5L^2}\sum_{\tau=0}^{t}\eta_\tau\alpha_\tau + 2\sqrt{5}L\tilde{\rho}^3\sum_{\tau=0}^{t+1}\eta_\tau^2\alpha_\tau.
    \end{aligned}
    \end{equation*}
    where
    $$
    \Tilde{\rho}:= \max\bigg\{\frac{1}{1-\sqrt{\frac{1+\rho}{2}}} \sqrt{\frac{1+\rho^2}{1-\rho^2}}, \frac{\rho}{1-\rho}\frac{1+\rho^2}{1-\rho^2}\bigg\}
    $$
    Note that $\Tilde{\rho}$ is greater than $0$.
\end{lemma}
\begin{proof}
    The first line is by Lemma \ref{lemma_one_step_consensus} and \ref{lemma_technical_summation} by taking $a_\tau=\frac{1}{n}\E\|\X^{\tau} - \Bar{\X}^{\tau}\|^2$, $b_\tau=\eta_\tau^2 (1+\rho^2) / (1-\rho^2)$, $c_\tau=1$ and $r=(1+\rho) / 2$ in Lemma \ref{lemma_technical_summation} directly.
    
    For the second line, by Lemma \ref{lemma_one_step_consensus} we get 
    \begin{equation*}
    \begin{split}
        \sqrt{\frac{1}{n}\E\|\X^{t+1} - \Bar{\X}^{t+1}\|^2} \leq&\sqrt{ \frac{1+\rho}{2}\frac{1}{n}\E\|\X^{t} - \Bar{\X}^{t}\|^2 + \eta_t^2 \frac{1+\rho^2}{1-\rho^2}\frac{1}{n}\E\|\hat{\Z}^{t} - \Bar{\hat{\Z}}^t\|^2} \\
        \leq & \sqrt{\frac{1+\rho}{2}}\sqrt{ \frac{1}{n}\E\|\X^{t} - \Bar{\X}^{t}\|^2} + \eta_t \sqrt{\frac{1+\rho^2}{1-\rho^2}}\sqrt{\frac{1}{n}\E\|\hat{\Z}^{t} - \Bar{\hat{\Z}}^t\|^2} \\
        \leq & \sqrt{\frac{1+\rho}{2}}\sqrt{ \frac{1}{n}\E\|\X^{t} - \Bar{\X}^{t}\|^2} + \eta_t \sqrt{\frac{1+\rho^2}{1-\rho^2}}
    \end{split}
    \end{equation*}
    where the second inequality is by $\sqrt{a+b}\leq\sqrt{a} + \sqrt{b}$ and third is by $\|\hat{\Z}^{t} - \Bar{\hat{\Z}}^t\|^2\leq n$. Now taking $a_\tau=\sqrt{\frac{1}{n}\E\|\X^{\tau} - \Bar{\X}^{\tau}\|^2}$, $b_\tau=\eta_\tau \sqrt{(1+\rho^2) / (1-\rho^2)}$, $c_\tau=\alpha_{\tau}$ and $r=\sqrt{(1+\rho) / 2}$ as in Lemma \ref{lemma_technical_summation} will give the first line of the result. Note that here $a_0=0$ due to the initialization of our algorithm.

    Now to the third line, again by Lemma \ref{lemma_one_step_consensus} we get 
    \begin{equation*}
    \begin{split}
        \sum_{t=0}^{T-1}\frac{1}{n}\E\|\Z^{t} - \Bar{\Z}^{t}\|^2 &\leq \Tilde{\rho} \sum_{t=0}^{T-1}\alpha_t^2 \frac{1}{n}\E\|\U^t - \Bar{\U}^t\|^2 \\
        \sum_{t=0}^{T-1}\alpha_t^2 \frac{1}{n}\E\|\U^t - \Bar{\U}^t\|^2 &\leq \Tilde{\rho} \sum_{t=0}^{T-1}\alpha_t^2 \frac{1}{n}\E\|\V^{t+1} - \V^t\|^2
    \end{split}
    \end{equation*}
    by using Lemma \ref{lemma_one_step_consensus} for two times. Also since
    \begin{equation*}
    \begin{split}
        \sum_{t=0}^{T-1}\alpha_t^2 \frac{1}{n}\E\|\V^{t+1} - \V^t\|^2 &\leq (10 \sigma^2+ 5 L^2) \sum_{t=0}^{T-1}\alpha_t^2 + 2\sum_{t=0}^{T}\alpha_t^2 \frac{1}{n}\E\left\|\X^{t}-\bar{\X}^{t}\right\|^2 \\
        \sum_{t=0}^{T}\alpha_t^2 \frac{1}{n}\E\left\|\X^{t}-\bar{\X}^{t}\right\|^2 &\leq \Tilde{\rho} \sum_{t=0}^{T} \alpha_t^2 \eta_t^2
    \end{split}
    \end{equation*}
    where for the third line we again use Lemma \ref{lemma_one_step_consensus}. Combining all above equations gives the second line of the theorem.

    As for the forth line, note that from Lemma \ref{lemma_one_step_consensus} and $\sqrt{a+b}\leq\sqrt{a} + \sqrt{b}$, we have:
    \begin{equation*}
    \begin{aligned}
         \sqrt{\frac{1}{n}\E\|\Z^{t+1} - \Bar{\Z}^{t+1}\|^2} &\leq \sqrt{\frac{1+\rho}{2}} \sqrt{\frac{1}{n}\E\|\Z^{t} - \Bar{\Z}^{t}\|^2} + \alpha_t\sqrt{\frac{1+\rho^2}{1-\rho^2}}  \sqrt{\frac{1}{n}\E\|\U^t - \Bar{\U}^t\|^2},\\
         \sqrt{\frac{1}{n}\E\|\U^{t+1} - \Bar{\U}^{t+1}\|^2} &\leq \sqrt{\frac{1+\rho}{2}} \sqrt{\frac{1}{n}\E\|\U^{t} - \Bar{\U}^{t}\|^2} + \sqrt{\frac{1+\rho^2}{1-\rho^2}}  \sqrt{\frac{1}{n}\E\|\V^{t+1} - \V^t\|^2}, \\
         \sqrt{\frac{1}{n}\E\|\V^{t+1} - \V^t\|^2}&\leq \sqrt{10 \sigma^2+5 L^2} +\sqrt{5}L \left[\sqrt{\frac{1}{n}\E\left\|\X^{t+1}-\bar{\X}^{t+1}\right\|^2}+\sqrt{\frac{1}{n}\E\left\|\X^{t}-\bar{\X}^{t}\right\|^2}\right].
    \end{aligned}
    \end{equation*}
    Repeating the proof of the third line gives the fourth line.
\end{proof}

Now we are ready to show our final convergence for constant stepsizes, which we restate as follows:
\begin{theorem}\label{thm_fix_step_appendix}
    Suppose Assumptions \ref{assump_l_smooth} and \ref{assump_bdd_variance} hold, also take $\alpha_t = \alpha T^{-1/2}$ and $\eta_t=\eta T^{-3/4}$ for any $\alpha, \eta>0$, the update of Algorithm \ref{algo_decen_normalized_averaged_grad_tracking} satisfies:
    \begin{equation*}
    \begin{split}
        &\frac{1}{T} \sum_{t=0}^{T-1}\E\|\nabla f(\bar{x}^t)\| \leq \mathcal{O}\bigg( (\frac{\Delta_0}{\eta } + \frac{2L\eta}{\alpha} + \frac{2\sigma\sqrt{\alpha}}{\sqrt{n}})\frac{1}{T^{1/4}} + \frac{\Tilde{\rho}^2\sqrt{10\sigma^2+5L^2}\alpha}{T^{1/2}} \bigg), \\
        &\frac{1}{T} \sum_{t=0}^{T-1}\E\|\Bar{z}^t - \nabla f(\Bar{x}^t)\| \leq \mathcal{O}\bigg( (\frac{2L\eta}{\alpha} + \frac{2\sigma\sqrt{\alpha}}{\sqrt{n}})\frac{1}{T^{1/4}} + 2L\tilde{\rho}\eta\frac{1}{T^{1/2}} \bigg), \\
        &\frac{1}{T} \sum_{t=0}^{T-1}\frac{1}{n}\left[\|\X^{t} - \Bar{\X}^{t}\|^2 + \|\Z^{t} - \Bar{\Z}^{t}\|^2\right] \leq \mathcal{O}\bigg( \Tilde{\rho}\frac{\eta}{T^{1/2}}+ \Tilde{\rho}^2(10 \sigma^2+ 5L^2)\frac{\alpha^4}{T} + 2\Tilde{\rho}^3 \frac{\alpha^4\eta^2}{T^{5/2}} \bigg).
    \end{split}
    \end{equation*}
    The above three bounds correspond to stationarity, approximation to gradient and consensus errors. Note that we hide some higher-order terms in $\mathcal{O}$.
\end{theorem}
\begin{proof}%[Proof of Theorem \ref{thm_fix_step}]
    By Lemma \ref{lemma_descent}, we have 
    \begin{equation*}
    \begin{aligned}
        \eta_t\E\|\nabla f(\bar{x}^t)\|&\leq \E[f(\bar{x}^{t}) - f(\bar{x}^{t+1})] + 2\eta_t \E\|\Bar{z}^t - \nabla f(\bar{x}^t)\| + \E[\frac{\eta_t}{n}\sum_{i=1}^{n}\|z_{i}^{t} - \Bar{z}^t\|] + \frac{\eta_t^2 L}{2} \\
        &\leq \E[f(\bar{x}^{t}) - f(\bar{x}^{t+1})] + 2\eta_t \E\|\Bar{z}^t - \nabla f(\bar{x}^t)\| + \eta_t\sqrt{\frac{1}{n}\sum_{i=1}^{n}\E\|z_{i}^{t} - \Bar{z}^t\|^2} + \frac{\eta_t^2 L}{2}
    \end{aligned}
    \end{equation*}
    where in the second equality we used $\E X^2\geq (\E X)^2 $.

    Now sum up from $t=0$ to $T-1$ and using Lemma \ref{lemma_descent} and \ref{lemma_concensus_summation}, we get
    \begin{equation}\label{eq_thm_fix_temp1}
    \begin{aligned}
        \sum_{t=0}^{T-1}\eta_t\E\|\nabla f(\bar{x}^t)\| \leq& \Delta_0 + 2\sum_{t=0}^{T-1}\eta_t \E\|\Bar{z}^t - \nabla f(\bar{x}^t)\| + \sum_{t=0}^{T-1}\eta_t\sqrt{\frac{1}{n}\sum_{i=1}^{n}\E\|z_{i}^{t} - \Bar{z}^t\|^2} + \frac{L}{2}\sum_{t=0}^{T-1}\eta_t^2 \\
        \leq& \Delta_0 + 2\sum_{t=0}^{T-1}\eta_t \bigg(L\sum_{\tau=1}^{t}\beta_{(\tau+1):t}\eta_\tau + L\tilde{\rho}\sum_{\tau=0}^{t}\beta_{(\tau+1):t}\alpha_{\tau}\eta_{\tau} + \sigma\sqrt{\frac{1}{n}\sum_{\tau=0}^{t}\beta_{(\tau+1):t}^2\alpha_{\tau}^2}\bigg) \\ &+\Tilde{\rho}^2\sqrt{10\sigma^2+5L^2}\sum_{t=0}^{T-1}\eta_t\alpha_t + 2\sqrt{5}L\tilde{\rho}^3\sum_{t=0}^{T}\eta_t^2\alpha_t+ \frac{L}{2}\sum_{t=0}^{T-1}\eta_t^2
    \end{aligned}
    \end{equation}

    Now we inspect the each of the terms on the right hand side. By our choice of $\eta_t$ and $\alpha_t$, it's straightforward to verify that
    \begin{equation}\label{eq_thm_fix_temp2}
    \begin{aligned}
        \sum_{t=1}^{T}\alpha_t\eta_t &= \frac{\alpha\eta}{T^{1/4}},\ \sum_{t=1}^{T}\alpha_t\eta_t^2 =\frac{\alpha\eta^2}{T},\ \sum_{t=1}^{T}\eta_t^2 = \frac{\eta^2}{T^{1/2}}\\
    \end{aligned}
    \end{equation}
    and
    \begin{equation*}
    \begin{aligned}
        \sum_{t=0}^{T-1}\eta_t\sum_{\tau=1}^{t}\beta_{(\tau+1):t}\eta_\tau &= \sum_{t=0}^{T-1}\frac{\eta}{T^{3/4}}\sum_{\tau=1}^{t}(1-\frac{\alpha}{T^{1/2}})^{t-\tau}\frac{\eta}{T^{3/4}}\\
        &= \frac{\eta^2}{T^{3/2}} \sum_{t=0}^{T-1}\sum_{\tau=1}^{t}(1-\frac{\alpha}{T^{1/2}})^{t-\tau} \\
        &= \frac{\eta^2}{T^{3/2}} \sum_{t=0}^{T-1}(1-\frac{\alpha}{T^{1/2}})^{t} \frac{T^{1/2}}{\alpha}[(1-\frac{\alpha}{T^{1/2}})^{-t+1} - 1] \\
        &\leq \frac{\eta^2}{\alpha T} \sum_{t=0}^{T-1} (1-\frac{\alpha}{T^{1/2}})\leq \frac{\eta^2}{\alpha}
    \end{aligned}
    \end{equation*}
    Similarly
    \begin{equation*}
    \begin{aligned}
        \sum_{t=0}^{T-1}\eta_t\sum_{\tau=0}^{t}\beta_{(\tau+1):t}\alpha_{\tau}\eta_{\tau} &\leq \frac{\eta^2}{T^{1/2}}\\
        \sum_{t=0}^{T-1}\eta_t\sqrt{\sum_{\tau=0}^{t}\beta_{(\tau+1):t}^2\alpha_{\tau}^2} &\leq \eta\sqrt{\alpha}
    \end{aligned}
    \end{equation*}

    Now plugging everything back we get:
    \begin{equation*}
    \begin{aligned}
        \frac{\eta}{T^{3/4}}\sum_{t=0}^{T-1}\E\|\nabla f(\bar{x}^t)\|
        &\leq \Delta_0 + 2L\frac{\eta^2}{\alpha} + 2L\tilde{\rho}\frac{\eta^2}{T^{1/2}}+ 2\frac{\sigma}{\sqrt{n}}\eta\sqrt{\alpha} \\ & +\Tilde{\rho}^2\sqrt{10\sigma^2+5L^2}\frac{\alpha\eta}{T^{1/4}} + 2\sqrt{5}L\tilde{\rho}^3\frac{\alpha\eta^2}{T} + \frac{L}{2}\frac{\eta^2}{T^{1/2}}
    \end{aligned}
    \end{equation*}
    i.e. 
    \begin{equation*}
    \begin{aligned}
        \frac{1}{T}\sum_{t=0}^{T-1}\E\|\nabla f(\bar{x}^t)\|
        &\leq (\frac{\Delta_0}{\eta } + \frac{2L\eta}{\alpha} + \frac{2\sigma\sqrt{\alpha}}{\sqrt{n}})\frac{1}{T^{1/4}} + (2L\tilde{\rho}\eta+\frac{L\eta}{2})\frac{1}{T^{3/4}}\\ & +\frac{\Tilde{\rho}^2\sqrt{10\sigma^2+5L^2}\alpha}{T^{1/2}} + \frac{2\sqrt{5}L\tilde{\rho}^3\alpha^2\eta}{T^{7/4}}
    \end{aligned}
    \end{equation*}
    The first line of the theorem is obtained by neglecting the higher-order terms. Note that we also proved the second line of the theorem since \eqref{eq_thm_fix_temp1} already contains the bound for $\|\Bar{z}^t-\nabla f(\Bar{x}^t)\|$.

    It remains to bound the consensus error (third line), which follows directly from Lemma \ref{lemma_concensus_summation} and \eqref{eq_thm_fix_temp2}.
\end{proof}

% \begin{remark}\label{rmk_fix_step_linear_speedup}
%     To make $1/T\sum_{t=0}^{T-1}\E\|\nabla f(\Bar{x}^t)\|\leq\epsilon$, we need $T=\tilde{\mathcal{O}}(1/\epsilon^4)$, which matches the lower bound as in (cite Carmon's lower bound paper).
    
%     To achieve the linear speedup~\citep{lian2017can} (cite Cho's paper here), one concrete choice of stepsize would be $\alpha=n^{1/2}$ and $\eta=n^{1/4}$ so that we have the convergence rate of:
%     \begin{equation*}
%         \frac{1}{T} \sum_{t=0}^{T-1}\E\|\nabla f(\bar{x}^t)\| \leq \mathcal{O}\bigg( \frac{\Delta_0 + L + \sigma}{n^{1/4}T^{1/4}} + \frac{\Tilde{\rho}^2(\sigma+L)n^{1/2}}{T^{1/2}} \bigg)
%     \end{equation*}
%     which implies that to make $1/T\sum_{t=0}^{T-1}\E\|\nabla f(\Bar{x}^t)\|\leq\epsilon$, the sample complexity on each agent is of order $\mathcal{O}(T) = \mathcal{O}(1/(n\epsilon^4))$ provided that $T$ is sufficiently large, which indicates that we have a clear linear speedup~\citep{lian2017can}. Note that the approximation error $\|\Bar{z}^t-\nabla f(\Bar{x}^t)\|$ also enjoys this linear speedup effect. This choice of parameter requires prior knowledge of the total number of nodes, yet still doesn't require global information about the loss function or the topological information about the communication graph.
% \end{remark}

We also present the result when we don't fix the total number of iterations in advance. We have the following useful technical lemma. Most of the result in this Lemma is from Lemma 11 in \cite{hubler2023parameter}.

\begin{lemma}\label{lemma_technical_series}
    Let $q\in(0, 1)$, $p\geq0$ and $t>0$. Further let positive integers $a, b$ s.t. $2\leq a\leq b$, then we have that for any $\alpha>0$,
\begin{equation*}
\prod_{t=a}^b\left(1-\alpha t^{-q}\right) \leq \exp \left(\frac{\alpha}{1-q}\left(a^{1-q}-b^{1-q}\right)\right).
\end{equation*}

If in addition $p\geq q$, we have
\begin{equation*}
    \sum_{t=a}^b t^{-p} \prod_{\tau=a}^t\left(1-\alpha \tau^{-q}\right) \leq \frac{(a-1)^{q-p} \exp \left(\alpha\frac{a^{1-q}-(a-1)^{1-q}}{1-q}\right)-b^{q-p} \exp \left(\alpha\frac{a^{1-q}-b^{1-q}}{1-q}\right)}{(\alpha+(p-q) b^{q-1})}.
\end{equation*}
and in particular,
\begin{equation*}
    \sum_{t=a}^b t^{-p} \prod_{\tau=a}^t\left(1-\alpha \tau^{-q}\right) \leq \frac{(a-1)^{q-p} }{\alpha} \exp \left(\alpha\frac{a^{1-q}-(a-1)^{1-q}}{1-q}\right) = \mathcal{O}(\frac{a^{q-p}}{\alpha}).
\end{equation*}

If further in addition $p<1$, $\alpha \geq 2$, we have
\begin{equation*}
\sum_{t=2}^b t^{-p} \prod_{\tau=t+1}^b\left(1-\alpha\tau^{-q}\right) \leq \frac{2}{\alpha}\exp \left(\frac{\alpha}{1-q}\right)(b+1)^{q-p}
\end{equation*}
\end{lemma}
\begin{proof}
    For the first equation we get
\begin{equation*}
\prod_{t=a}^b\left(1-\alpha t^{-q}\right) \leq \exp \left(-\alpha \sum_{\tau=a}^b t^{-q}\right) \leq \exp \left(-\alpha \int_a^{b+1} t^{-q} d t\right)=\exp \left(\frac{\alpha }{1-q}\left(a^{1-q}-(b+1)^{1-q}\right)\right).
\end{equation*}

Now for the second line, using the above result we get
\begin{equation}\label{eq_lemma_tech_series_1}
\begin{aligned}
    \sum_{t=a}^b t^{-p} \prod_{\tau=a}^t\left(1-\alpha \tau^{-q}\right) &\leq \exp \left(\frac{\alpha a^{1-q}}{1-q}\right) \sum_{t=a}^b t^{-p} \exp \left(-\frac{\alpha (t+1)^{1-q}}{1-q}\right) \\
    & \leq \exp \left(\frac{\alpha a^{1-q}}{1-q}\right) \int_{a-1}^{b} t^{-p} \exp \left(-\frac{\alpha t^{1-q}}{1-q}\right)dt \\
    & = \exp \left(\frac{\alpha a^{1-q}}{1-q}\right) \int_{a-1}^{b} t^{-q} t^{q-p} \exp \left(-\frac{\alpha t^{1-q}}{1-q}\right)dt
\end{aligned}
\end{equation}
The above integral can be calculated by integration by parts, specifically:
\begin{equation*}
\begin{aligned}
    &\int_{a-1}^{b} t^{q-p} t^{-q} \exp \left(-\frac{\alpha t^{1-q}}{1-q}\right) dt \\
    &= \left[ -\frac{t^{q-p}}{\alpha} \exp \left(-\frac{\alpha t^{1-q}}{1-q}\right) \right]_{t=a-1}^{t=b} + (q - p) \int_{a-1}^{b} \frac{t^{q-p-1}}{\alpha} \exp\left( -\frac{\alpha t^{1-q}}{1-q} \right) dt.
\end{aligned}
\end{equation*}
Finally, since the integrand is monotonically decreasing and $p\geq q$, we have
$$
(q - p)\int_{a-1}^{b} t^{q-p-1}\exp\left( -\frac{\alpha t^{1-q}}{1-q} \right) dt \leq (q - p)b^{q-1}\int_{a-1}^{b} t^{-p}\exp\left( -\frac{\alpha t^{1-q}}{1-q} \right) dt
$$
which is exactly the integral we started with. The results in the second and third lines are thus by rearranging terms.

Now for the last line of the lemma, and we use the similar technique as the second line, both adopted from the proof of Lemma 10 of \cite{hubler2023parameter}. We have
\begin{equation*}
\begin{aligned}
    \sum_{t=2}^b t^{-p} \prod_{\tau=t+1}^b\left(1-\alpha\tau^{-q}\right) &\leq \exp\left(-\alpha\sum_{\tau=1}^{b}\tau^{-q}\right)\sum_{t=2}^{b}t^{-p}\exp\left(\alpha\sum_{\tau=1}^{t}\tau^{-q}\right) \\
    &\leq \exp\left(-\alpha\int_{1}^{b+1}\tau^{-q}d\tau\right)\sum_{t=2}^{b}t^{-p}\exp\left(\alpha\int_{0}^{t}\tau^{-q}d\tau\right) \\
    &\leq \exp\left( \alpha\frac{1-(b+1)^{1-q}}{1-q}\right)\sum_{t=2}^{b}t^{-p}\exp\left(\alpha\frac{t^{1-q}}{1-q}\right) 
\end{aligned}
\end{equation*}
Note that the summation in the last line above is the same as \eqref{eq_lemma_tech_series_1}. Repeat the proof of the second line gives the result.
\end{proof}

\begin{lemma}\label{lemma_technical_summation2}
    Suppose we take $\alpha_t = \alpha t^{-1/2}$ and $\eta_t=\eta t^{-3/4}$ ($\alpha_0=\eta_0=0$) with $0<\alpha<1$ and $\eta>0$, then we have 
    \begin{equation*}
    \begin{aligned}
        \sum_{t=1}^{T}\alpha_t\eta_t &\leq \mathcal{O}(\alpha\eta),\ \sum_{t=1}^{T}\alpha_t\eta_t^2 \leq \mathcal{O}(\alpha\eta^2),\ \sum_{t=1}^{T}\eta_t^2 \leq \mathcal{O}(\eta^2)\\
    \end{aligned}
    \end{equation*}
    and
    \begin{equation*}
    \begin{aligned}
        \sum_{t=0}^{T-1}\eta_t\sum_{\tau=1}^{t}\beta_{(\tau+1):t}\eta_\tau &\leq \mathcal{O}(\frac{\eta^2}{\alpha}\log(T))\\
        \sum_{t=0}^{T-1}\eta_t\sum_{\tau=0}^{t}\beta_{(\tau+1):t}\alpha_{\tau}\eta_{\tau} &\leq \mathcal{O}(\eta^2\log(T))\\
        \sum_{t=0}^{T-1}\eta_t\sqrt{\sum_{\tau=0}^{t}\beta_{(\tau+1):t}^2\alpha_{\tau}^2} &\leq \mathcal{O}(\sqrt{\alpha}\eta \log(T))
    \end{aligned}
    \end{equation*}
    where $\beta_t = 1-\alpha_t$.
\end{lemma}
\begin{proof}
    The first line is directly by the integral test of the series in the form $\sum_t t^{-p}$ with $p>1$.
    
    The proof of the latter three resembles the proof of Lemma 11 in \cite{hubler2023parameter}. We prove the second line of this Lemma as a show case and refer to Lemma 11 in \cite{hubler2023parameter} for the detail of the proof of the last two lines. For the second line, we have
    \begin{equation*}
    \begin{aligned}
        \sum_{t=0}^{T-1}\eta_t\sum_{\tau=1}^{t}\beta_{(\tau+1):t}\eta_\tau &= \eta^2 + \eta^2 \sum_{t=1}^{T-1}t^{-3/4}\sum_{\tau=2}^{t}\prod_{\xi=\tau+1}^{t}(1-\alpha\xi^{-1/2})\tau^{-3/4} \\
        & \leq \eta^2 + \eta^2 \sum_{t=1}^{T-1}t^{-3/4}\frac{2}{\alpha}\exp \left(2\alpha\right)(t+1)^{-1/4} \\
        &= \eta^2 + \frac{2\eta^2}{\alpha}\exp \left(2\alpha\right)\sum_{t=1}^{T-1}t^{-1} \leq \eta^2 + \frac{2\eta^2}{\alpha}\exp \left(2\alpha\right)\log(T)
    \end{aligned}
    \end{equation*}
    where we use Lemma \ref{lemma_technical_series} for the first inequality.
\end{proof}

Now we are ready to present the final convergence result. We restate the convergence theorem as follows:
\begin{theorem}\label{thm_dim_step_appendix}
    Suppose Assumptions \ref{assump_l_smooth} and \ref{assump_bdd_variance} hold, also take $\alpha_t = \alpha t^{-1/2}$ and $\eta_t=\eta t^{-3/4}$ for any $\eta>0$, the update of Algorithm \ref{algo_decen_normalized_averaged_grad_tracking} satisfies:
    \begin{equation*}
    \begin{split}
        &\frac{1}{T} \sum_{t=0}^{T-1}\E\|\nabla f(\bar{x}^t)\| \leq \Tilde{\mathcal{O}}(\frac{\Delta_0/\eta+L\eta/\alpha+L\Tilde{\rho}\eta+\sigma\sqrt{\alpha} / \sqrt{n} + \Tilde{\rho}^2(\sigma+L)\alpha+L\Tilde{\rho}^3 \eta\alpha+L\eta}{T^{1/4}}), \\
        &\frac{1}{T} \sum_{t=0}^{T-1}\E\|\Bar{z}^t - \nabla f(\Bar{x}^t)\| \leq \Tilde{\mathcal{O}}\bigg( \frac{L\eta/\alpha+L\Tilde{\rho}\eta+\sigma\sqrt{\alpha} / \sqrt{n}}{T^{1/4}} \bigg), \\
        &\frac{1}{T} \sum_{t=0}^{T-1}\frac{1}{n}\E\left[\|\X^{t} - \Bar{\X}^{t}\|^2 + \|\Z^{t} - \Bar{\Z}^{t}\|^2\right] \leq \mathcal{O}\bigg(\frac{\Tilde{\rho}^2(\sigma^2+ L^2 + \Tilde{\rho} \eta^2)}{T}\bigg).
    \end{split}
    \end{equation*}
    The above three bounds correspond to stationarity, approximation to gradient and consensus errors. Note that we hide logarithmic factors in $\Tilde{\mathcal{O}}$.
\end{theorem}
\begin{proof}%[Proof of Theorem \ref{thm_dim_step}]
    Same as the proof of Theorem \ref{thm_fix_step_appendix}, we get
    \begin{equation*}
    \begin{aligned}
        \sum_{t=0}^{T-1}\eta_t\E\|\nabla f(\bar{x}^t)\|
        \leq& \Delta_0 + 2L\sum_{t=0}^{T-1}\eta_t\sum_{\tau=1}^{t}\beta_{(\tau+1):t}\eta_\tau + 2L\tilde{\rho}\sum_{t=0}^{T-1}\eta_t\sum_{\tau=0}^{t}\beta_{(\tau+1):t}\alpha_{\tau}\eta_{\tau} + 2\frac{\sigma}{\sqrt{n}}\sum_{t=0}^{T-1}\eta_t\sqrt{\sum_{\tau=0}^{t}\beta_{(\tau+1):t}^2\alpha_{\tau}^2} \\ &+\Tilde{\rho}^2\sqrt{10\sigma^2+5L^2}\sum_{t=0}^{T-1}\eta_t\alpha_t + 2\sqrt{5}L\tilde{\rho}^3\sum_{t=0}^{T}\eta_t^2\alpha_t+ \frac{L}{2}\sum_{t=0}^{T-1}\eta_t^2
    \end{aligned}
    \end{equation*}

    Now we inspect the each of the terms on the right hand side. By our choice of $\eta_t$ and $\alpha_t$ and using Lemma \ref{lemma_technical_summation2} we get:
    \begin{equation*}
        \sum_{t=0}^{T-1}\eta_t\E\|\nabla f(\bar{x}^t)\| \leq \Tilde{\mathcal{O}}(\Delta_0+L\frac{\eta^2}{\alpha}+L\Tilde{\rho}\eta^2+\frac{\sigma}{\sqrt{n}}\sqrt{\alpha}\eta + \Tilde{\rho}^2(\sigma+L)\eta\alpha+L\Tilde{\rho}^3 \eta^2\alpha+L\eta^2)
    \end{equation*}
    where we hide the logarithmic factor in $\Tilde{\mathcal{O}}$.

    Now since $\frac{1}{T}\sum_{t=0}^{T-1}\E\|\nabla f(\bar{x}^t)\|\leq T^{-1/4}\sum_{t=0}^{T-1}t^{-3/4}\E\|\nabla f(\bar{x}^t)\|$, we yield the desired result in the theorem statement.
\end{proof}

\end{document}